\documentclass[11 pt]{amsart}
\usepackage[margin=3.3 cm]{geometry}
\usepackage{amsfonts}
\usepackage{amsmath}
\usepackage{amsthm}
\usepackage{amssymb}
\usepackage{amsmath,amscd}
\usepackage{enumerate}
\usepackage{subcaption}
\usepackage[pdffitwindow=false,
colorlinks=true,linkcolor=black,urlcolor=black,citecolor=black]{hyperref}
\usepackage{tabularx}
\usepackage{float}
\usepackage{tablefootnote}
\usepackage{multirow}

\newtheorem{theorem}{Theorem}[section]
\newtheorem{lemma}[theorem]{Lemma}
\theoremstyle{definition}

\newtheorem{proposition}[theorem]{Proposition}

\newtheorem{remark}[theorem]{Remark}
\numberwithin{equation}{section}

\newcommand{\trace}{\mathrm{tr}}

\newcommand{\diag}{\mathrm{diag}}

\title{Infinitely many non-collapsed steady Ricci solitons on complex line bundles}
\author{Hanci Chi}
\address{Department of Pure Mathematics\\ Xi'an Jiaotong-Liverpool University\\ Suzhou 215123\\ China}
\email{hanci.chi@xjtlu.edu.cn}
\begin{document}
\maketitle

\begin{abstract}
We construct a continuous 3-parameter family of non-shrinking Ricci solitons on complex line bundles $\mathcal{O}(k)$ over $\mathbb{CP}^{2m+1}$, where the base space is not necessarily  K\"ahler--Einstein. Each $\mathcal{O}(k)$ with $k\in [3,2m+1]$ admits at least one asymptotically conical (AC) Ricci-flat metric in this family. For each $\mathcal{O}(k)$ with $k\geq 3$, the family includes infinitely many asymptotically paraboloidal (AP) steady Ricci solitons.
\end{abstract}

\let\thefootnote\relax\footnote{


The author is supported by the NSFC (No. 12301078), the Natural Science Foundation of Jiangsu Province (BK-20220282), and XJTLU Research Development Fund (RDF-21-02-083).}
\section{Introduction}
Ricci solitons are candidates for blow-up singularities of the Ricci flow. Specifically, a shrinking Ricci soliton is a candidate for a Type I singularity while a steady Ricci soliton is a candidate for a Type II singularity. From Perelman's no-collapsing theorem \cite{perelman_entropy_2002}, Type II singularities must be non-collapsed. 

Complex line bundles $\mathcal{O}(k)$ over $\mathbb{CP}^n$ provide a rich source of examples for gradient Ricci solitons. The principal orbit of $\mathcal{O}(k)$ is the lens space $\mathbb{S}^{2n+1}/\mathbb{Z}_k$, a circle bundle over $\mathbb{CP}^n$. The integer $k$ classifies the circle bundle, corresponding to a multiple of the indivisible element in  $H^2(\mathbb{CP}^n,\mathbb{Z})$. This setup has led to numerous examples of Ricci solitons with various asymptotic behaviors. The B\'erard-Bergery metric on $\mathcal{O}(n+1)$ is an AC Ricci-flat metric; see \cite{berard-bergery_sur_1982}. By \cite{Cao_Elliptic_1996}, it is known that $\mathcal{O}(n+1)$ admits a 1-parameter family of steady K\"ahler--Ricci solitons. These solitons have collapsed volume growth. In particular, they are asymptotically cigar-paraboloidal (ACP). Expanding K\"ahler--Ricci solitons on $\mathcal{O}(k)$ with $k>n+1$ are found in \cite{feldman_rotationally_2003}. The $U(n+1)$-invariant ansatz is further generalized to complex line bundles over products of Fano manifolds in \cite{dancer2011ricci}. Recently, new steady Ricci solitons were found on $\mathcal{O}(k)$ independently by \cite{wink2023cohomogeneity}, \cite{stolarski2024steady} and \cite{appleton_family_2017}. It was proved in \cite{appleton_family_2017} that each $\mathcal{O}(k)$ with $k>n+1$ admits a non-collapsed AP steady Ricci soliton.

In this article, we examine complex line bundles $\mathcal{O}(k)$ over $\mathbb{CP}^{2m+1}$, where the symmetry condition is relaxed from $U(2m+2)$-invariance to $Sp(m+1)Sp(1)$-invariance, allowing the base space to be squashed. Specifically, the associated group triple is:
\begin{equation}
\label{eqn_group_triple}
(\mathsf{K},\mathsf{L},\mathsf{G})=([Sp(m)\Delta U(1)]\times \mathbb{Z}_k,Sp(m)U(1)U(1),Sp(m+1)U(1)).
\end{equation}
Identify $\mathbb{S}^{4m+3}/\mathbb{Z}_k$ as the quotient $\{\mathbf{q}\mid \mathbf{q}\in\mathbb{H}^{m+1}, |\mathbf{q}|=1\}/\{\mathbf{q}\sim e^{\frac{2\pi i}{k}}\mathbf{q}\}$. Each element $(A,z)\in Sp(m+1)U(1)$ acts on $\mathbb{S}^{4m+3}/\mathbb{Z}_k$ via $\mathbf{q}\mapsto A\mathbf{q}\bar{z}$. The isotropy subgroup of $(0,\dots,0,1)$ is
$$\mathsf{K}=\left\{\left(\diag(A,e^{\frac{2l\pi i}{k}}z),z\right)\mid z\in U(1), l=0,...,k-1\right\}.$$ The group action descends to $\mathbb{CP}^{2m+1}$, where the isotropy subgroup of $[0: ... :0:1]$ is 
$$\mathsf{H}=\left\{\left(\diag(A,z_1),z_2\right)\mid z_1,z_2\in U(1)\right\}.$$ In the special case where $k=1$, this construction yields the quaternionic Hopf fibration. 

The present work is closely related to, but fundamentally different from, our previous work \cite{chi_Non-Shrinking_2024}, where continuous 1-parameter families of non-collapsed steady Ricci solitons were constructed on Euclidean spaces and the quaternionic line bundle over $\mathbb{HP}^m$. Here we study complex line bundles $\mathcal{O}(k)\rightarrow\mathbb{CP}^{2m+1}$, where the principal orbit is $\mathbb{S}^{4m+3}/\mathbb{Z}_k$. Since the principal orbit is locally isometric to $\mathbb{S}^{4m+3}$, the cohomogeneity one dynamical system on $\mathcal{O}(k)$ is independent of the integer $k$. Nevertheless, different topology types lead to different smoothness conditions. Furthermore, our global analysis shows that the existence of non-collapsed solitons also depends on the topology of the underlying manifold, which is encoded by $k$.

For $k\in\{1,2\}$, we construct a compact invariant set to obtain the following theorem.

\begin{theorem}
\label{thm_1}
For $k\in \{1,2\}$, there exists a continuous $3$-parameter family of complete Ricci solitons $\{\xi(k,\theta,s_4,s_5)\mid \theta\in (0,\pi),s_4\geq 0, s_5\geq 0\}$ on $\mathcal{O}(k)$. In particular, 
\begin{enumerate}
\item
each $\xi(k,\theta,s_4,s_5)$ with $s_4,s_5>0$ is an AC non-Einstein expanding Ricci soliton;
\item
each $\xi(k,\theta,0,s_5)$ with $s_5>0$ is an asymptotically hyperbolic (AH) Einstein metric;
\item
each $\xi(k,\theta,0,0)$ is an asymptotically locally conical (ALC) Ricci-flat metric. The base of the asymptotic cone is the non-K\"ahler $\mathbb{CP}^{2m+1}$;
\item
each $\xi(k,\theta,s_4,0)$ with $s_4>0$ is an ACP steady Ricci soliton. The base of the asymptotic paraboloid is the non-K\"ahler $\mathbb{CP}^{2m+1}$.  
\end{enumerate}
\end{theorem}

The parameter $s_5$ controls the generalized mean curvature of the principal orbit. As the generalized mean curvature is freely scaled under homothetic change if $\epsilon=0$, the parameter $s_5$ is not needed and we set $s_5=0$ for the steady case. The parameter $s_4$ is related to the second derivative of the potential function, and the Ricci soliton is Einstein if $s_4=0$. The parameter $\theta$ controls how $\mathbb{CP}^{m+1}$ is squashed as a twistor space over $\mathbb{HP}^m$. Solutions with $\theta=0$ admit $U(2m+2)$-symmetry and we have the Fubini--Study $\mathbb{CP}^{m+1}$. The limiting case where $\theta=\pi$ blows up the base space $\mathbb{HP}^m$. The cohomogeneity one steady soliton equation of this limiting case is equivalent to that of a complex line bundle over an irreducible $\mathbb{CP}^1$, whose principal orbit is $\mathbb{S}^3/\mathbb{Z}_k$. 

If $k\geq 3$, our techniques for proving Theorem \ref{thm_1} become obsolete. Nevertheless, we generalize \cite[Theorem 5.1]{appleton_family_2017} and obtain the following existence theorems.
\begin{theorem}
\label{thm_2-1}
For each $k\geq 3$ and $\theta\in (0,\pi)$, there exists a $\beta_{k,\theta}\geq 0$ such that $\{\xi(k,\theta,s_4,s_5)\mid s_4\geq \beta_{k,\theta}, s_5> 0\}$ is a continuous $2$-parameter family of complete expanding Ricci solitons. In particular,
\begin{enumerate}
\item
each $\xi(k,\theta,s_4,s_5)$ with $s_4>\beta_{k,\theta}$ is an AC non-Einstein Ricci soliton;
\item
each $\xi(k,\theta,\beta_{k,\theta},s_5)$ is an AH negative Einstein metric if $\beta_{k,\theta}=0$, an AC non-Einstein Ricci soliton if $\beta_{k,\theta}>0$.
\end{enumerate}
\end{theorem}

\begin{theorem}
\label{thm_2-2}
For each $k\geq 3$ and $\theta\in (0,\pi)$, there exists an $\alpha_{k,\theta}\geq 0$ such that $\{\xi(k,\theta,s_4,0)\mid s_4\geq \alpha_{k,\theta}\}$ is a continuous $1$-parameter family of complete steady Ricci solitons. In particular,
\begin{enumerate}
\item
each $\xi(k,\theta,\alpha_{k,\theta},0)$ is an AC or an ALC Ricci-flat metric if $\alpha_{k,\theta}=0$;
\item
each $\xi(k,\theta,\alpha_{k,\theta},0)$ is an AP non-Einstein steady soliton if $\alpha_{k,\theta}>0$. The base of the asymptotic paraboloid is the the Jensen $\mathbb{S}^{4m+3}/\mathbb{Z}_k$.
\end{enumerate}
\end{theorem}
For comparison, we list previous examples and new Ricci solitons in Theorem \ref{thm_1}-\ref{thm_2-2} in the following table.

\tiny
\begin{table}[H]
\centering
  \begin{tabular}{l l l l| l }
   & Metric & Asymptotic & Base & Source \\
\hline
\hline
$\xi(k,0,0,0)$  & \multirow{2}{*}{almost-Hermitian Ricci-flat}        & \multirow{2}{*}{ALC}   & \multirow{2}{*}{Fubini--Study $\mathbb{CP}^{2m+1}$} & \multirow{5}{*}{\cite{berard-bergery_sur_1982}}\\
$1\leq k\leq 2m+1$  &         &  &  & \\
\cline{1-4}
$\xi(2m+2,0,0,0)$     &almost-K\"ahler Ricci-flat & AC  & standard $\mathbb{S}^{4m+3}/\mathbb{Z}_{2m+2}$ &\\ 
\cline{1-4}
$\xi(k,0,0,s_5)$     &\multirow{2}{*}{negative Einstein}& \multirow{2}{*}{AH}  &  &\\ 
$k\geq 1$, $s_5>0$  &         &  &  & \\
\hline
$\xi(2m+2,0,s_4,0)$     & \multirow{2}{*}{steady K\"ahler--Ricci soliton}         & \multirow{2}{*}{ACP} & \multirow{2}{*}{Fubini--Study $\mathbb{CP}^{2m+1}$} &  \multirow{2}{*}{\cite{Cao_Elliptic_1996}} \\
$s_4>0$  &         &  &  & \\
\hline
$\xi(k,0,s_4,k-(2m+2))$   & \multirow{2}{*}{expanding K\"ahler--Ricci soliton}         & \multirow{2}{*}{AC} & &  \multirow{2}{*}{\cite{feldman_rotationally_2003}}\\
$k\geq 2m+3$, $s_4>0$ &&&&\\  
\hline
$\xi(k,0,s_4,0)$     & \multirow{2}{*}{steady Ricci soliton}         & \multirow{2}{*}{ACP/AP} & \multirow{2}{*}{Fubini--Study $\mathbb{CP}^{2m+1}$} &  \multirow{2}{*}{\cite{stolarski2024steady}}\\
$k\geq 1$, $s_4>0$ &&&&\\  
\hline
$\xi(k,0,s_4,0)$   & \multirow{2}{*}{steady Ricci soliton} & \multirow{2}{*}{ACP} & \multirow{2}{*}{Fubini--Study $\mathbb{CP}^{2m+1}$} &\multirow{4}{*}{\cite{wink2023cohomogeneity}} \\
$ 1\leq k\leq \sqrt{2m+2}$, $s_4>0$ &&&&\\  
\cline{1-4}
$\xi(k,0,s_4,s_5)$   & \multirow{2}{*}{expanding Ricci soliton} & \multirow{2}{*}{AC} & \\
$1\leq k\leq \sqrt{2m+2}$, $s_4>0$, $s_5>0$ &&&&\\  
\hline
$\xi(k,0,s_4,0)$     & \multirow{2}{*}{steady Ricci soliton}         & \multirow{2}{*}{ACP} & \multirow{2}{*}{Fubini--Study $\mathbb{CP}^{2m+1}$} &  \multirow{6}{*}{\cite{appleton_family_2017}}\\
$1\leq k\leq 2m+2$, $s_4>0$ &&&&\\  
\cline{1-4}
$\xi(k,0,s_4,0)$     & \multirow{2}{*}{steady Ricci soliton}        & \multirow{2}{*}{ACP} & \multirow{2}{*}{Fubini--Study $\mathbb{CP}^{2m+1}$}  \\
$k\geq 2m+3$, $s_4>\alpha_{k,0}$ &&&&\\  
\cline{1-4}
$\xi(k,0,\alpha_{k,0},0) $   & \multirow{2}{*}{steady Ricci soliton}         & \multirow{2}{*}{AP} & \multirow{2}{*}{standard $\mathbb{S}^{4m+3}/\mathbb{Z}_k$}  \\
$k\geq 2m+3$ &&&&\\  
\hline
\\
\hline
$\xi(k,\theta,0,0)$    & Ricci-flat    & ALC & non-K\"ahler $\mathbb{CP}^{2m+1}$& \multirow{1}{*}{Theorem \ref{thm_1}}  \\ 
\cline{1-4}
$\xi(k,\theta,0,s_5)$ & negative Einstein & AH & &   $k\in \{1,2\}$,\\ 
\cline{1-4}
$\xi(k,\theta,s_4,0)$ & steady Ricci soliton  & ACP &  non-K\"ahler $\mathbb{CP}^{2m+1}$ & $\theta\in (0,\pi)$,  \\ 
\cline{1-4}
$\xi(k,\theta,s_4,s_5)$ & expanding Ricci soliton  & AC & & $s_4>0$, $s_5>0$  \\ 
\hline
\\
\hline
$\xi(k,\theta,s_4,s_5)$    & expanding Ricci soliton    & AC & &  \multirow{1}{*}{Theorem \ref{thm_2-1}}\\ 
\cline{1-4}
\multirow{2}{*}{$\xi(k,\theta,\beta_{k,\theta},s_5)$}  &if $\beta_{k,\theta}>0$:  expanding Ricci soliton & AC & & $k\geq 3$, $\theta\in (0,\pi)$,  \\ 
\cline{2-4}
& if $\beta_{k,\theta}=0$: negative Einstein & AH & & $s_4>\beta_{k,\theta}$, $s_5>0$  \\ 
\hline
\\
\hline
\multirow{2}{*}{$\xi(k,\theta,\alpha_{k,\theta},0)$} & \multirow{2}{*}{if $\alpha_{k,\theta}=0$: Ricci-flat}  & AC &  Jensen $\mathbb{S}^{4m+3}/\mathbb{Z}_k$ & Theorem \ref{thm_2-2} \\ 
\cline{3-4}
 &  & ALC & non-K\"ahler $\mathbb{CP}^{2m+1}$ & $k\geq 3$,\\ 
\cline{1-4}
$\xi(k,\theta,\alpha_{k,\theta},0)$ & if $\alpha_{k,\theta}>0$: steady Ricci soliton & AP & Jensen $\mathbb{S}^{4m+3}/\mathbb{Z}_k$ & $\theta\in(0,\pi)$,\\  
\cline{1-4}
$\xi(k,\theta,s_4,0)$ & steady Ricci soliton & ACP & non-K\"ahler $\mathbb{CP}^{2m+1}$ & $s_4>\alpha_{k,\theta}$  \\  
\hline
\end{tabular}
\caption{}
\label{tab_xi_curve}
\end{table}
\normalsize

Compared to Theorem \ref{thm_1}, the asymptotics for steady solitons in Theorem \ref{thm_2-2} are less settled, as they depend on the value of $\alpha_{k,\theta}$. The boundary cases where $\theta \in \{0, \pi\}$, however, are better understood in the previous examples. Specifically, the metrics $\xi(2m+2,0,0,0)$ and $\xi(2,\pi,0,0)$ are AC Ricci-flat metrics, with the asymptotic cones' bases being the standard orbifolds $\mathbb{S}^{4m+3}/\mathbb{Z}_{2m+2}$ and $\mathbb{S}^{3}/\mathbb{Z}_2$, respectively; see \cite{berard-bergery_sur_1982}. The existence of $\alpha_{k,0}$ and $\alpha_{k,\pi}$ are established in \cite[Theorem 5.1]{appleton_family_2017}. In particular, we have $\alpha_{k,0} > 0$ if $k \geq 2m+3$, and $\alpha_{k,\pi} > 0$ if $k \geq 3$. Thanks to these results, we can further obtain the following existence theorem for some non-collapsed steady Ricci solitons.
%
%
%
%
%

\begin{theorem}
\label{thm_3}
On each $\mathcal{O}(k)$ with $3\leq k\leq 2m+1$, there exists at least one AC Ricci-flat metric, whose base of the cone is the Jensen $\mathbb{S}^{4m+3}/\mathbb{Z}_k$. On each $\mathcal{O}(k)$ with $k\geq 3$, there exist infinitely many AP steady Ricci solitons whose base of the paraboloid is the Jensen $\mathbb{S}^{4m+3}/\mathbb{Z}_k$. 
\end{theorem} 

This article is structured as follows. We first present the cohomogeneity one Ricci soliton equation and introduce a coordinate transformation to compactify the phase space. In the present setting, the compactified system has a non-hyperbolic critical point corresponding to the singular orbit, and therefore the local existence requires a central manifold analysis. This is in contrast with our previous work \cite{chi_Non-Shrinking_2024}, where the corresponding critical point is hyperbolic.

For the global analysis, a key feature of the present construction is that the topology parameter $k$ leads to different dynamical behaviors of the integral curves. When $k\in\{1,2\}$, the integral curves can be controlled by extending the compact invariant set constructed in \cite{chi_Non-Shrinking_2024}, which yields complete collapsed steady Ricci solitons in Theorem \ref{thm_1}. On the other hand, non-collapsed solitons occur for $k\geq 3$, where the integral curves becomes substantially different. In this case, we generalize \cite[Theorem 5.1]{appleton_family_2017} and apply a shooting argument to establish the existence and asymptotic behavior of these solitons, proving Theorems \ref{thm_2-1}-\ref{thm_3}.

\section{Cohomogeneity one equation}
We first describe the cohomogeneity one structure of the complex line bundles $\mathcal{O}(k)$ and present the corresponding Ricci soliton equations. Let $Q$ be the inner product on $\mathfrak{g}$ that generates the round metric on $\mathbb{S}^{4m+3}$. We have a $Q$-orthogonal decomposition of the isotropy representation $\mathfrak{g}/\mathfrak{k}$ into three irreducible summands. Consider the ansatz 
\begin{equation}
\label{eqn_coho1metric}
g=dt^2+a^2(t)\left.Q\right|_{\mathbf{1}}+b^2(t)\left.Q\right|_{\mathbf{2}}+c^2(t)\left.Q\right|_{\mathbf{4m}},
\end{equation}
where $b$ and $c$ are metric components for the base space $\mathbb{CP}^{2m+1}$ and $a$ is for the $\mathbb{S}^1$-fiber. Since the principal orbit is locally isometric to the sphere, the intrinsic Ricci curvature have the same expressions as in the previous setting. Therefore, the cohomogeneity one Ricci soliton equations take the same form as in \cite{chi_Non-Shrinking_2024}, namely,
\begin{equation}
\label{eqn_cohomogeneity_one_soliton_equation}
\begin{split}
\frac{\ddot{a}}{a}-\left(\frac{\dot{a}}{a}\right)^2&=-\left(\frac{\dot{a}}{a}+2\frac{\dot{b}}{b}+4m\frac{\dot{c}}{c}\right)\frac{\dot{a}}{a}+2 \frac{a^2}{b^4}+4 m \frac{a^2}{c^4}+\frac{\dot{a}}{a}\dot{f}+\frac{\epsilon}{2},\\
\frac{\ddot{b}}{b}-\left(\frac{\dot{b}}{b}\right)^2&=-\left(\frac{\dot{a}}{a}+2\frac{\dot{b}}{b}+4m\frac{\dot{c}}{c}\right)\frac{\dot{b}}{b}+\frac{4}{b^2}-2\frac{a^2}{b^4}+ 4m \frac{b^2}{c^4}+\frac{\dot{b}}{b}\dot{f}+\frac{\epsilon}{2},\\
\frac{\ddot{c}}{c}-\left(\frac{\dot{c}}{c}\right)^2&=-\left(\frac{\dot{a}}{a}+2\frac{\dot{b}}{b}+4m\frac{\dot{c}}{c}\right)\frac{\dot{c}}{c}+\frac{4m+8}{c^2}-2\frac{a^2}{c^4}-4\frac{b^2}{c^4}+\frac{\dot{c}}{c}\dot{f}+\frac{\epsilon}{2},\\
\ddot{f}+\frac{\epsilon}{2}&=\frac{\ddot{a}}{a}+2\frac{\ddot{b}}{b}+4m\frac{\ddot{c}}{c}.
\end{split}
\end{equation}
Despite the identical dynamical system, the topology of the underlying manifold, encoded by $k$, changes the smoothness conditions at the singular orbit and global behavior of admissible integral curves.

We set $f(0)=0$ without loss of generality. By \cite{buzano_initial_2011}, the potential function $f$ is an even function in $t$ near $0$. As the $\mathbb{S}^1$-fiber collapses at the singular orbit, we have $\trace(L)=\frac{1}{t}+O(t)$ as $t\to 0$. Therefore, we have $2\ddot{f}(0)=C$ by\eqref{eqn_potential_equation}. As the radius for the $\mathbb{S}^1$-fiber of $\mathbb{S}^{4m+3}/\mathbb{Z}_k$ is $\frac{1}{k}$, we require $\dot{a}(0)=k$ for smooth extension and we consider the initial condition
\begin{equation}
\label{eqn_initial_condition_G/L}
\lim_{t\to 0} (a,b,c,f,\dot{a},\dot{b},\dot{c})=\left(0,b_0,c_0,0,k,0,0\right),\quad b_0, c_0>0.
\end{equation}
With each fixed $k$, parameters $b_0$, $c_0$ and $C$ are free. The equation \eqref{eqn_cohomogeneity_one_soliton_equation} is invariant under homothety change if $\epsilon=0$. Therefore, the above initial condition has two free parameters for the steady case. Note that the singular orbit is not round if $b_0\neq c_0$. 

By \cite[Theorem 20.1]{hamilton_formations_1993}, the equality
\begin{equation}
\label{eqn_potential_equation}
\ddot{f}+(\trace(L)-\dot{f})\dot{f}-\epsilon f=C
\end{equation}
holds for some constant $C$. Canceling all second order terms using \eqref{eqn_cohomogeneity_one_soliton_equation} and \eqref{eqn_potential_equation}, we obtain the following equation:
\begin{equation}
\label{eqn_cohomogneity_one_conservation_law}
\begin{split}
&\trace(L^2)+r_s+\frac{n-1}{2}\epsilon-(\dot{f}-\trace(L))^2=C+\epsilon f.
\end{split}
\end{equation}
By \cite[Corollary 2.3]{chen_strong_2009}, it can be deduced that $C+\epsilon f<0$ for a non-Einstein Ricci soliton. Adopting the coordinate change below, the dynamical system \eqref{eqn_cohomogeneity_one_soliton_equation} is transformed to one that is associated to a polynomial vector field, and the phase space is compactified by the Einstein equations.

Consider $d\eta=(\trace(L)-\dot{f})dt$. Define
\begin{equation}
\label{eqn_new_variable_for_new_coordinate}
\begin{split}
&X_1=\frac{\frac{\dot{a}}{a}}{\trace(L)-\dot{f}},\quad X_2=\frac{\frac{\dot{b}}{b}}{\trace(L)-\dot{f}},\quad X_3=\frac{\frac{\dot{c}}{c}}{\trace(L)-\dot{f}},\\
&Z_1=\frac{a^2}{b^2},\quad Z_2=\frac{\frac{1}{b^2}}{(\trace(L)-\dot{f})^2},\quad Z_3=\frac{\frac{b^2}{c^4}}{(\trace(L)-\dot{f})^2},\quad Z_4=\frac{\frac{1}{c^2}}{(\trace(L)-\dot{f})^2},\quad W=\frac{1}{(\trace(L)-\dot{f})^2}.
\end{split}
\end{equation}
Define functions on $\eta$
\begin{equation}
\label{eqn_curvature function}
\begin{split}
&R_1=2 Z_1Z_2+4m Z_1Z_3,\\
&R_2=4 Z_2-2 Z_1 Z_2+ 4m Z_3,\\
&R_3=(4m+8)Z_4-2 Z_1 Z_3-4 Z_3,\\
&R_s=R_1+2R_2+4mR_3=8Z_2+4m(4m+8)Z_4-8m Z_3-2 Z_1 Z_2-4m Z_1 Z_3,\\
&H=X_1+2X_2+4mX_3,\quad G=X_1^2+2X_2^2+4mX_3^2.
\end{split}
\end{equation}
Let $'$ denote the derivative with respect to $\eta$. The soliton equation \eqref{eqn_cohomogeneity_one_soliton_equation} becomes a polynomial system
\begin{equation}
\label{eqn_Polynomial_soliton_equation}
\begin{bmatrix}
X_1\\
X_2\\
X_3\\
Z_1\\
Z_2\\
Z_3\\
Z_4\\
W
\end{bmatrix}'=
V(X_1,X_2,X_3,Y_1,Y_2,Y_3,W):=\begin{bmatrix}
V_1\\
V_2\\
V_3\\
V_4\\
V_5\\
V_6\\
V_7\\
V_8
\end{bmatrix}
=\begin{bmatrix}
X_1(G-\frac{\epsilon}{2}W-1)+R_1+\frac{\epsilon}{2}W\\
X_2(G-\frac{\epsilon}{2}W-1)+R_2+\frac{\epsilon}{2}W\\
X_3(G-\frac{\epsilon}{2}W-1)+R_3+\frac{\epsilon}{2}W\\
2Z_1(X_1-X_2)\\
2Z_2(G-\frac{\epsilon}{2}W-X_2)\\
2Z_3(G-\frac{\epsilon}{2}W+X_2-2X_3)\\
2Z_4(G-\frac{\epsilon}{2}W-X_3)\\
2W(G-\frac{\epsilon}{2}W)
\end{bmatrix}
\end{equation}
Define the polynomial $Q:= G+R_s+(n-1)\frac{\epsilon}{2}W-1$.
The conservation law \eqref{eqn_cohomogneity_one_conservation_law} becomes
\begin{equation}
\label{eqn_new_cohomogneity_one_conservation_law}
\begin{split}
\frac{C+\epsilon f}{(\trace(L)-\dot{f})^2}= Q.
\end{split}
\end{equation}
As discussed above, the quantity $C+\epsilon f$ (and hence $Q$) remains negative for a non-Einstein Ricci soliton. We focus on the following $7$-dimensional algebraic set.
$$\mathcal{RS}:=\{Q\leq 0\}\cap\{H\leq 1\}\cap\{W\geq 0\}\cap \{Z_1,Z_2,Z_3,Z_4\geq 0\}\cap\{Z_4^2=Z_2Z_3\}.$$
Since 
\begin{equation}
\label{eqn_characteriaztion_of_non_trivial_soliton}
\begin{split}
\begin{bmatrix}
Q\\
H
\end{bmatrix}'=\begin{bmatrix}
2Q(G-\frac{\epsilon}{2}W)+\epsilon(H-1)W\\
(H-1)\left(G-\frac{\epsilon}{2}W-1\right)+Q
\end{bmatrix},
\end{split}
\end{equation}
it is clear that $\mathcal{RS}$ is flow-invariant, a slight generalization to \cite[Proposition 3.1]{dancer_non-kahler_2009}. For the expanding case, we fix the homothety by setting $\epsilon=1$. For $\epsilon=0$, the homothety change is transformed to translation on the variable $\eta$ in the new coordinate. Therefore, different trajectories to system \eqref{eqn_Polynomial_soliton_equation} in $\mathcal{RS}$ represent different homothety families of Ricci solitons.

The coordinate transformation above is different from the $XYW$-coordinates used in \cite{chi_Non-Shrinking_2024}, where
\begin{equation}
\label{eqn_Y_to_Z}
Y_1=\sqrt{Z_1},\quad Y_2=\sqrt{Z_2},\quad Y_3=\sqrt{Z_3}.
\end{equation}
The motivation for introducing the $XZW$-coordinates is that the linearization of the compactified system at $p_0$ has the same positive eigenvalue, which is particularly convenient for the local analysis near the singular orbit.

Local solutions that satisfy the initial condition \eqref{eqn_initial_condition_G/L} is transformed integral curves that emanate $p_0:=(1,0,0,0,0,0,0,0,0)$. The positiviteness and finiteness of $b_0$ and $c_0$ requires the limit
\begin{equation}
\label{eqn_new initial condition for 0th derivative-1}
\lim\limits_{\eta\to-\infty}\sqrt{\frac{Z_3}{Z_2}}>0
\end{equation}
to exist. If $\epsilon=1$, the zeroth order initial condition $\lim\limits_{t\to 0}(b,c)=(b_0,c_0)$ becomes
\begin{equation}
\label{eqn_new initial condition for 0th derivative-2}
\lim\limits_{\eta\to-\infty} \frac{W}{Z_2}=b_0^2,\quad \lim\limits_{\eta\to-\infty} \frac{W}{Z_3}=\frac{c_0^2}{b_0},\quad \lim\limits_{\eta\to-\infty} \frac{W}{Z_4}=c_0^2.
\end{equation}
The first order initial condition $\lim\limits_{t\to 0}(\dot{a},\dot{b},\dot{c})=(k,0,0)$ becomes
\begin{equation}
\label{eqn_new initial condition for 1st derivative}
\lim\limits_{\eta\to-\infty}\sqrt{\frac{Z_1}{Z_2}}=k,\quad \lim\limits_{\eta\to-\infty} \frac{X_2}{\sqrt{Z_2}}=0,\quad \lim\limits_{\eta\to-\infty} \frac{X_3}{\sqrt{Z_4}}=0.
\end{equation}
Therefore, as in the previous setting, the construction of complete Ricci solitons reduces to finding integral curves emanating from $p_0$ and defined on $\mathbb{R}$.

We end this section by listing some important invariant subsets of lower dimensions below. 
\begin{enumerate}
\item
$\mathcal{E}:=\mathcal{RS}\cap\{Q=0\}\cap\{H=1\}$

As discussed above, equalities $Q=0$ and $H=1$ mean that $f$ is identically zero and $C=0$. The subsystem restricted on $\mathcal{E}$ is hence the cohomogeneity one non-positive Einstein equation. 

\item
$\mathcal{RS}_{\text{steady}}:=\mathcal{RS}\cap\{W=0\}$

The variable $W$ is decoupled if $\epsilon=0$. Hence a solution to the subsystem restricted on $\mathcal{RS}_{\text{steady}}$ represents a steady Ricci soliton even though the function $\frac{1}{(\trace(L)-\dot{f})^2}$ does not identically vanish.

\item
$\mathcal{RF}:=\mathcal{RS}_{\text{steady}}\cap \mathcal{RS}_{\text{Einstein}}$

The subsystem restricted on $\mathcal{RF}$ is the cohomogeneity one Ricci-flat equation. 

\item
$\mathcal{RS}_{\text{FS}}:=\mathcal{RS}\cap\{Z_2-Z_3=0\}\cap\{X_3-X_3=0\}$

Trajectories on this set represent metrics whose singular orbit is the Fubini--Study $\mathbb{CP}^{2m+1}$. The K\"ahler class of $\mathbb{CP}^{2m+1}$ is $k$ multiple of indivisible integral cohomology class in $H^2(\mathbb{CP}^{2m+1},\mathbb{Z})$. For integral curves that are not in $\mathcal{RS}_{\text{FS}}$, the base space is not K\"ahler.

\item
$\mathcal{RS}_{\text{KE}}:=\mathcal{RS}_{\text{FS}}\cap\{X_2^2=Z_1Z_2\}\cap\{(4m+4)Z_2+\frac{\epsilon}{2}W-X_2(1+X_1)=0\}$

Trajectories on this set represent K\"ahler--Ricci solitons; see \cite{dancer2011ricci}. Note that the base space must be the Fubini--Study $\mathbb{CP}^{2m+1}$.

\item
$\mathcal{RS}_{\text{round}}:=\mathcal{RS}\cap\{Z_1=1\}\cap\{X_1-X_2=0\}$

Trajectories on this set represent metrics with a principal orbit as a round $\mathbb{S}^3$-bundle over $\mathbb{HP}^m$.

\item
$\mathcal{RS}_{m=0} := \mathcal{RS}_{\text{steady}} \cap \{X_3, Z_3 = 0\}$.

Although integral curves in $\mathcal{RS}_{m=0}$ are geometrically degenerate when viewed as $(4m+4)$-dimensional Ricci solitons, they can be interpreted as four-dimensional steady Ricci solitons on complex line bundles over $\mathbb{CP}^{1}$. Indeed, setting $m=\epsilon=0$ in \eqref{eqn_cohomogeneity_one_soliton_equation} and ignoring the function $c$ yields precisely the system for four-dimensional steady Ricci solitons as in \cite[(2.1)--(2.3)]{appleton_family_2017}. Applying the coordinate transformation \eqref{eqn_new_variable_for_new_coordinate}, this system is equivalent to the restriction of \eqref{eqn_Polynomial_soliton_equation} to $\mathcal{RS}_{\text{steady}} \cap \mathcal{RS}_{\text{FS}}$ with $m=0$, which in turn coincides with the dynamical system on $\mathcal{RS}_{m=0}$. It is important to note that, this $m=0$ subsystem produces only steady solitons. The four-dimensional expanders studied in \cite{wang2025computer, donovan2025cohomogeneity} are hence not covered.
\end{enumerate}

\section{Local Analysis}
The linearization at $p_0$ is 
\begin{equation}
\label{eqn_linearization_for_CP2m+1}
\begin{bmatrix}
2&0&0&0&0&0&0&0\\
0&0&0&0& 4 & 4m &0&\frac{\epsilon}{2}\\
0&0&0&0&0& -4 & 4m+8 &\frac{\epsilon}{2}\\
0&0&0&2&0&0&0&0\\
0&0&0&0&2&0&0&0\\
0&0&0&0&0&2&0&0\\
0&0&0&0&0&0&2&0\\
0&0&0&0&0&0&0&2\\
\end{bmatrix}.
\end{equation}
The linearization has two zero eigenvalues.
The other eigenvalue of $2$ has multiplicity of $6$. Unstable eigenvectors are:
\scriptsize
$$
w_1=\begin{bmatrix}
-(4m+2)(2m+2)\\
2m+2\\
2m+2\\
0\\
1\\
1\\
1\\
0
\end{bmatrix},
w_2=\begin{bmatrix}
-4\\
2\\
0\\
0\\
1\\
0\\
0\\
0
\end{bmatrix},
w_3=\begin{bmatrix}
-4(m+1)^2\sqrt{2}\\
2\sqrt{2}\\
(m+2)\sqrt{2}\\
0\\
\sqrt{2}\\
0\\
\frac{\sqrt{2}}{2}\\
0
\end{bmatrix},
w_4=\begin{bmatrix}
-1\\
0\\
0\\
0\\
0\\
0\\
0\\
0
\end{bmatrix},
w_5=\begin{bmatrix}
-(4m+2)\frac{\epsilon}{2}\\
\frac{\epsilon}{2}\\
\frac{\epsilon}{2}\\
0\\
0\\
0\\
0\\
2
\end{bmatrix},
w_6=\begin{bmatrix}
0\\
0\\
0\\
1\\
0\\
0\\
0\\
0
\end{bmatrix}.
$$
\normalsize

Since $p_0$ is not hyperbolic, the Hartman--Grobman Theorem does not directly apply. We apply center manifold theory for the local analysis near $p_0$; see \cite{carr2012applications} and \cite{perko2013differential}. Consider the linear transformation
\begin{equation}
\label{eqn_linear coordinate change}
\tilde{X}_1=X_1-1,\quad B:=X_2-2Z_2-2mZ_3-\frac{\epsilon}{4}W,\quad C:=X_3+2Z_3-(2m+4)Z_4-\frac{\epsilon}{4}W.
\end{equation}
The dynamical system \eqref{eqn_Polynomial_soliton_equation} is topologically conjugate to 
\begin{equation}
\label{eqn_Polynomial_soliton_equation_2}
\begin{bmatrix}
\tilde{X}_1\\
B\\
C\\
Z_1\\
Z_2\\
Z_3\\
Z_4\\
W
\end{bmatrix}'=F(\tilde{X}_1,B,C,Z_1,Z_2,Z_3,Z_4,W)=\begin{bmatrix}
V_1\\
F_B\\
F_C\\
V_4\\
V_5\\
V_6\\
V_7\\
V_8
\end{bmatrix}=\begin{bmatrix}
V_1\\
V_2-2V_5-2mV_6-\frac{\epsilon}{4}V_8\\
V_3+2V_6-(2m+4)V_7-\frac{\epsilon}{4}V_8\\
V_4\\
V_5\\
V_6\\
V_7\\
V_8
\end{bmatrix}
\end{equation}
where $V_i$ are functions in \eqref{eqn_Polynomial_soliton_equation} with $X_i$ being functions defined by \eqref{eqn_linear coordinate change}. The linearization of \eqref{eqn_Polynomial_soliton_equation_2} at $p_0$ is 
\begin{equation}
\label{eqn_linearization_for_CP2m+1}
\begin{bmatrix}
2&0&0&0&0&0&0&0\\
0&0&0&0&0&0&0&0\\
0&0&0&0&0&0&0&0\\
0&0&0&2&0&0&0&0\\
0&0&0&0&2&0&0&0\\
0&0&0&0&0&2&0&0\\
0&0&0&0&0&0&2&0\\
0&0&0&0&0&0&0&2\\
\end{bmatrix}.
\end{equation}
By \cite[Theorem 2, Page 161]{perko2013differential}, the dynamical system \eqref{eqn_Polynomial_soliton_equation_2} is topologically conjugate to 
\begin{equation}
\label{eqn_Polynomial_soliton_equation_3}
\begin{bmatrix}
\tilde{X}_1\\
B\\
C\\
Z_1\\
Z_2\\
Z_3\\
Z_4\\
W
\end{bmatrix}'=\begin{bmatrix}
2\tilde{X}_1\\
F_B(\tilde{x}_1,B,C,z_1,z_2,z_3,z_4,w)\\
F_C(\tilde{x}_1,B,C,z_1,z_2,z_3,z_4,w)\\
2Z_1\\
2Z_2\\
2Z_3\\
2Z_4\\
2W
\end{bmatrix},
\end{equation}
where lower case letters are functions of $(B,C)$ that satisfy
\begin{align}
\label{eqn_mafan_1}
&\begin{bmatrix}
\partial_B \tilde{x}_1& \partial_C \tilde{x}_1\\
\partial_B z_1& \partial_C z_1\\
\partial_B z_2& \partial_C z_2\\
\partial_B z_3& \partial_C z_3\\
\partial_B z_4& \partial_C z_4\\
\partial_B w  & \partial_C w
\end{bmatrix}
\begin{bmatrix}
F_B(\tilde{x}_1,B,C,z_1,z_2,z_3,z_4,w)\\
F_C(\tilde{x}_1,B,C,z_1,z_2,z_3,z_4,w)
\end{bmatrix}=\begin{bmatrix}
V_1(\tilde{x}_1,B,C,z_1,z_2,z_3,z_4,w)\\
V_4(\tilde{x}_1,B,C,z_1,z_2,z_3,z_4,w)\\
V_5(\tilde{x}_1,B,C,z_1,z_2,z_3,z_4,w)\\
V_6(\tilde{x}_1,B,C,z_1,z_2,z_3,z_4,w)\\
V_7(\tilde{x}_1,B,C,z_1,z_2,z_3,z_4,w)\\
V_8(\tilde{x}_1,B,C,z_1,z_2,z_3,z_4,w)
\end{bmatrix}\\
&\begin{bmatrix}
\tilde{x}_1\\
z_1\\
z_2\\
z_3\\
z_4\\
w
\end{bmatrix}(0,0)=\begin{bmatrix}
0\\
0\\
0\\
0\\
0\\
0\\
\end{bmatrix},\quad
\begin{bmatrix}
\partial_B \tilde{x}_1& \partial_C \tilde{x}_1\\
\partial_B z_1& \partial_C z_1\\
\partial_B z_2& \partial_C z_2\\
\partial_B z_3& \partial_C z_3\\
\partial_B z_4& \partial_C z_4\\
\partial_B w  & \partial_C w
\end{bmatrix}(0,0)=\begin{bmatrix}
0&0\\
0&0\\
0&0\\
0&0\\
0&0\\
0&0\\
\end{bmatrix}
\end{align}
and can be approximated by power series in $(B,C)$.

With functions $F_B$ and $F_C$ explicitly known in \eqref{eqn_Polynomial_soliton_equation_2}, we claim that
\begin{equation}
\label{eqn_all_vanish_1}
z_2=0, z_3=0, z_4=0, w=0
\end{equation}
as power series of $(B,C)$. Suppose otherwise, we have
$$
\begin{bmatrix}
F_B\\
F_C
\end{bmatrix}\sim \begin{bmatrix}
4 z_2+4m z_3 + \frac{\epsilon}{2}w\\
(4m + 8)z_4 - 4 z_3 + \frac{\epsilon}{2}w
\end{bmatrix},\quad \begin{bmatrix}
V_5\\
V_6\\
V_7\\
V_8
\end{bmatrix}\sim \begin{bmatrix}
2z_2\\
2z_3\\
2z_4\\
2w
\end{bmatrix}
$$
By the last four lines in \eqref{eqn_mafan_1}, the lowest power on the left hand side is strictly larger than the lowest power on the right hand side. Therefore, we must have power series \eqref{eqn_all_vanish_1}. Furthermore, we have 
$$
\begin{bmatrix}
F_B\\
F_C
\end{bmatrix}\sim \begin{bmatrix}
2\tilde{x}_1B\\
2\tilde{x}_1C
\end{bmatrix},\quad \begin{bmatrix}
V_1\\
V_4
\end{bmatrix}\sim \begin{bmatrix}
2\tilde{x}_1\\
2z_1
\end{bmatrix}
$$

Note that the lowest powers of $\tilde{x}_1B$ and $\tilde{x}_1 C$ are at least $3$. From the first two lines in \eqref{eqn_mafan_1}, we further conclude that $\tilde{x}_1=0$ and $z_1=0$ as power series of $(B,C)$. Therefore, the dynamical system \eqref{eqn_Polynomial_soliton_equation} is topologically conjugate to 
\begin{equation}
\label{eqn_Polynomial_soliton_equation_3}
\begin{bmatrix}
\tilde{X}_1\\
B\\
C\\
Z_1\\
Z_2\\
Z_3\\
Z_4\\
W
\end{bmatrix}'=\begin{bmatrix}
2\tilde{X}_1\\
B(2B^2+4m C^2)\\
C(2B^2+4m C^2)\\
2Z_1\\
2Z_2\\
2Z_3\\
2Z_4\\
2W
\end{bmatrix},
\end{equation}
whose general solutions are
\begin{equation}
\label{eqn_partly linear solution}
\begin{bmatrix}
\tilde{X}_1\\
B\\
C\\
Z_1\\
Z_2\\
Z_3\\
Z_4\\
W
\end{bmatrix}=\begin{bmatrix}
u_1e^{2\eta}\\
B_0\frac{1}{\sqrt{a_1-\eta}}\\
C_0\frac{1}{\sqrt{a_2-\eta}}\\
u_4e^{2\eta}\\
u_5e^{2\eta}\\
u_6e^{2\eta}\\
u_7e^{2\eta}\\
u_8e^{2\eta}\\
\end{bmatrix}.
\end{equation}
In particular, integral curves with non-zero $B_0$ or $C_0$ are center manifolds of $p_0$. These integral curves are geometrically degenerate. Specifically, by \eqref{eqn_new initial condition for 1st derivative} and the existence of \eqref{eqn_new initial condition for 0th derivative-2}, we have 
\begin{equation}
\begin{split}
\lim\limits_{\eta\to-\infty}\frac{B}{\sqrt{Z_2}}&=\lim\limits_{\eta\to-\infty}\frac{X_2-2Z_2-2mZ_3-\frac{\epsilon}{4}W}{\sqrt{Z}_2}=0,\\
\lim\limits_{\eta\to-\infty}\frac{C}{\sqrt{Z_4}}&=\lim\limits_{\eta\to-\infty}\frac{X_3+2Z_3-(2m+4)Z_4-\frac{\epsilon}{4}W}{\sqrt{Z_4}}=0.
\end{split}
\end{equation}
By \eqref{eqn_partly linear solution}, the above limits hold only if $B_0=C_0=0$. Therefore, we only consider integral curves in the unstable manifold of $p_0$. By \cite[Chapter 13, Theorem 4.5]{coddington_theory_1955} and the analysis above, there is no ambiguity to let $\xi(s_1,s_2,s_3,s_4,s_5,s_6)$ denote the integral curve of \eqref{eqn_Polynomial_soliton_equation}, where
\begin{equation}
\label{eqn_genearllinarized_Pcp}
\begin{split}
&\xi(s_1,s_2,s_3,s_4,s_5,s_6)\\
&= p_0+s_1e^{2\eta}w_1+s_2e^{2\eta}w_2+s_3e^{2\eta}w_3+s_4e^{2\eta}w_4+s_5e^{2\eta}w_5+s_6e^{2\eta}w_6+O(e^{(2+\delta)\eta}).
\end{split}
\end{equation}

The first three parameters control how the principal orbit is squashed. We fix $s_1^2+s_2^2+s_3^2=1$ to mod out the scalar multiplication of the vector $\sum_{i=1}^{6}s_iw_i$. As $p_0$ is at the tip of the constraint cone $\{Z_4^2=Z_2Z_3\}$, the normalized initial velocity must be tangent to the cone. Since
\begin{equation}
Z_2\sim \left(s_1+s_2+\sqrt{2}s_3\right)e^{2\eta},\quad Z_3\sim s_1 e^{2\eta},\quad Z_4\sim \left(s_1+\frac{\sqrt{2}}{2}s_3\right)e^{2\eta}
\end{equation}
near $p_0$, together with the constraint $Z_2Z_3=Z_4^2$, we have 
\begin{equation}
\label{eqn_parameter restriction}
s_1+s_2+\sqrt{2}s_3\geq 0,\quad s_1\geq 0,\quad  s_1+\frac{\sqrt{2}}{2}s_3\geq 0,\quad  s_3^2=2s_1s_2.
\end{equation}
We further impose the condition that $s_2+2s_3\geq 0$ so that $Z_2-Z_3\geq 0$ initially. Combining all the above conditions, the parameters $(s_1,s_2,s_3)$ are chosen from the circular arc
\begin{equation}
\label{eqn_parameter circle}
\begin{split}
&\left\{(s_1,s_2,s_3)\mid s_1^2+s_2^2+s_3^2=1,\quad s_3^2=2s_1s_2,\quad s_1,s_2,s_3\geq 0 \right\}\\
&=\left\{(s_1,s_2,s_3)\mid (s_1-s_2)^2+2s_3^2=1,\quad s_1+s_2=1,\quad s_3\geq 0 \right\}.
\end{split}
\end{equation}
On the other hand, the smoothness condition requires $\lim\limits_{t\to 0} \dot{a}= k$. By \eqref{eqn_new initial condition for 1st derivative}, we must set $s_6=k^2 (s_1+s_2+\sqrt{2}s_3).$ Parametrize \eqref{eqn_parameter circle} by
\begin{equation}
\label{eqn_s1s2s3 in theta}
s_1=\frac{1+\cos(\theta)}{2},\quad s_2=\frac{1-\cos(\theta)}{2},\quad s_3=\frac{1}{\sqrt{2}}\sin(\theta),\quad \theta\in[0,\pi].
\end{equation}
We write 
$$w(\theta,k)= \frac{1+\cos(\theta)}{2}w_1+\frac{1-\cos(\theta)}{2}w_2+\frac{1}{\sqrt{2}}\sin(\theta)w_3+k^2\left(1+\sin(\theta)\right) w_6.$$

All unstable eigenvectors are perpendicular to $(\nabla W)(p_0)$ except $w_4$. Steady Ricci solitons are represented by integral curves with $s_5=0$. Expanding Ricci solitons are represented by integral curves with $s_5> 0$. A straightforward computation shows that
$$
(\nabla Q)(p_0)=\begin{bmatrix}
2\\
0\\
0\\
0\\
8\\
-8m\\
4m(4m+8)\\
(n-1)\frac{\epsilon}{2}
\end{bmatrix}.
$$
Note that $w_4$ is the only unstable eigenvectors that is not perpendicular to $(\nabla Q)(P_{\mathbb{CP}^{2m+1}})$ or $(\nabla H)(P_{\mathbb{CP}^{2m+1}})$. By \eqref{eqn_genearllinarized_Pcp_2}, we have
\begin{equation}
\label{eqn_Q and 1-H}
Q\sim -2s_4 e^{2\eta},\quad 1-H\sim s_4 e^{2\eta}.
\end{equation}
Therefore, 
\begin{equation}
\label{eqn_linearized Q-H+1}
\frac{\ddot{f}}{(\trace(L)-\dot{f})^2}=Q-H+1\sim -s_4 e^{2\eta}
\end{equation}
as $\eta\to-\infty$. The parameter $s_4$ controls the initial condition $\ddot{f}(0)$ and we set $s_4\geq 0$. If integral curves with $s_4$ represent cohomogeneity one Einstein metrics.

In summary, we consider the following $4$-parameter family of integral curve such that
\begin{equation}
\label{eqn_genearllinarized_Pcp_2}
\begin{split}
\xi(k, \theta, s_4,s_5)= p_0+e^{2\eta}w(\theta,k)+s_4e^{2\eta}w_4+s_5e^{2\eta}w_5+O(e^{(2+\delta)\eta}),\quad \theta\in [0,\pi], k\in \mathbb{Z}^+, s_4,s_5\geq 0.
\end{split}
\end{equation}
For $\theta\in [0,\pi)$, these integral curves are identified as local solutions on a tubular neighborhood around $\mathbb{CP}^{2m+1}$. For $\theta=\pi$, the integral curves are in $\mathcal{RS}_{m=0}$, and they represent local solutions near $\mathbb{CP}^1$. The extra parameter $k$ creates the illusion of an additional degree of freedom to extend the local solution smoothly to the base space. However, for each fixed $\mathcal{O}(k)$, the smoothness condition only allows varying three parameters from $(\theta,s_4,s_5)$, as predicted in \cite{eschenburg_initial_2000}. 

\begin{remark}
Unlike the 1-parameter family of  K\"ahler--Ricci solitons on $\mathbb{C}^{2m+2}$ studied in \cite{cao1997limits}, the Ricci solitons represented by \eqref{eqn_genearllinarized_Pcp_2} have mixed signs in their sectional curvatures. Specifically, consider the normal-tangential sectional curvatures $K_i:=X_i-X_i^2-R_i-\frac{\epsilon}{2}W$. We observe that $K_1>0$, while $K_2,K_3<0$ initially along each $\xi(k,\theta,s_4,s_5)$. The non-negativity of sectional curvatures significantly constrains the topology of steady gradient Ricci solitons; for a detailed discussion, see the recent developments in \cite{deng2024fundamental}.
\end{remark}

\section{Global analysis}
In this section, we first construct a compact invariant set $\mathcal{F}$ to prove the completeness of Ricci solitons in Theorem \ref{thm_1}; see \eqref{eqn_setF}. Each $\xi(k, \theta, s_4,s_5)$ with $k\geq 3$ is not in $\mathcal{F}$ initially. Therefore, we generalize \cite[Theorem 5.1]{appleton_family_2017} to establish the existence of $\beta_{\theta,k}$ and $\alpha_{\theta,k}$ in Theorem \ref{thm_2-1}-\ref{thm_2-2}. By the continuous dependence of parameters, we also establish the existence of complete non-collapsed steady Ricci solitons in Theorem \ref{thm_3}.

Define the function 
$$
F_{l}:=X_2-X_1+l\left(\sqrt{\frac{Z_2}{Z_1}}-\sqrt{Z_1Z_2}\right)
$$
for some constant $l$. We have 
\begin{equation}
\label{eqn_derivative of Fl}
\begin{split}
\langle \nabla F_l,  V\rangle &=F_l\left(G-\frac{\epsilon}{2}W-1+2l\sqrt{Z_1Z_2}\right)\\
&\quad +\left(l\sqrt{\frac{Z_2}{Z_1}}(1-X_1)+(4-2l^2)Z_2+4m Z_3 \right)(1-Z_1).
\end{split}
\end{equation}
On $\mathcal{RS}\cap \{F_l=0\}\cap \{Z_2-Z_3\geq 0\}$, we have 
\begin{equation}
\label{eqn_conservation with Fl=0}
\begin{split}
1&\geq G+R_s\\
&= X_1^2+2\left(X_1-l\left(\sqrt{\frac{Z_2}{Z_1}}-\sqrt{Z_1Z_2}\right)\right)^2+4mX_3^2\\
&\quad +8Z_2+4m(4m+8)Z_4-8m Z_3-2 Z_1 Z_2-4m Z_1 Z_3
\end{split}
\end{equation}

\begin{proposition}
\label{prop_F2}
The inequality $\left.\langle \nabla F_2,  V\rangle \right|_{\{F_2=0\}} \geq 0$ holds on 
$$\mathcal{RS}\cap \{1-Z_1\geq 0\}\cap \{F_2=0\}\cap \{Z_2-Z_3\geq 0\}\cap \{X_1\geq 0\}$$
\end{proposition}
\begin{proof}
By inequalities $1-Z_1\geq 0$ and $Z_2-Z_3\geq 0$, the inequality \eqref{eqn_conservation with Fl=0} with $l=2$ becomes
\begin{equation}
\begin{split}
1&\geq X_1^2+2\left(X_1-2\left(\sqrt{\frac{Z_2}{Z_1}}-\sqrt{Z_1Z_2}\right)\right)^2 +8Z_2-2 Z_1 Z_2\\
&= 2\left(X_1-2\sqrt{\frac{Z_2}{Z_1}}+\sqrt{Z_1Z_2}\right)^2+(X_1+2\sqrt{Z_1Z_2})^2\\
&\geq (X_1+2\sqrt{Z_1Z_2})^2.
\end{split}
\end{equation}
Therefore, the inequality $1\geq X_1+2\sqrt{Z_1Z_2}$ holds on $\mathcal{RS}\cap \{1-Z_1\geq 0\}\cap \{F_2=0\}\cap \{Z_2-Z_3\geq 0\}\cap \{X_1\geq 0\}$. By \eqref{eqn_derivative of Fl}, we have 
\begin{equation}
\begin{split}
\left.\langle \nabla F_2,  V\rangle \right|_{\{F_2=0\}}&=\left(2\sqrt{\frac{Z_2}{Z_1}}(1-X_1)-4 Z_2+4m Z_3 \right)(1-Z_1)\\
&\geq \left(1-X_1-2\sqrt{Z_1Z_2}\right)2\sqrt{\frac{Z_2}{Z_1}}(1-Z_1)\\
&\geq 0.
\end{split}
\end{equation}
The proof is complete.
\end{proof}


\subsection{$k\leq 2$}

We proceed to prove that each $\xi(\theta,k,s_4,s_5)$ with $k\in\{1,2\}$ is defined on $\mathbb{R}$. Define 
\begin{equation}
\label{eqn_setF}
\begin{split}
\mathcal{F}&:=\mathcal{RS}\cap \{1-Z_1\geq 0\}\cap \left\{F_2\geq 0\right\}\\
&\quad \cap \{Z_2-Z_3\geq 0\}\cap \left\{2(\sqrt{Z_2}-\sqrt{Z_3})+X_3-X_2\geq 0\right\}\cap \{X_1,X_2,X_3\geq 0\}.
\end{split}
\end{equation}

As discussed in the introduction, the cases $k\in\{1,2\}$ exhibit relatively simpler dynamical behavior, similar to the setting considered in \cite{chi_Non-Shrinking_2024}. In this case, the compact set constructed above is sufficient to obtain complete collapsed steady Ricci solitons.

\begin{lemma}
\label{lem_invariant_setB}
The set $\mathcal{F}$ is compact and invariant.
\end{lemma}
\begin{proof}
With inequalities $1-Z_1\geq 0$ and $Z_2-Z_3\geq 0$, it is clear that $\mathcal{F}$ is compact. Furthermore, we have $R_2,R_3\geq 0$. Therefore, the vector field $V$ restricted on faces $\mathcal{F}\cap \left\{X_i=0\right\}$
points inward. 

For the face $\mathcal{F}\cap \{Z_2-Z_3=0\}$ and $\mathcal{F}\cap \{Z_1=1\}$, we have 
\begin{equation}
\label{eqn_Z2 Z3}
\left.\left\langle\nabla \left( Z_2-Z_3\right),V\right\rangle\right|_{\{Z_2-Z_3=0\}}=4Z_2(X_3-X_2)\geq -8Z_2(\sqrt{Z_2}-\sqrt{Z_3})=0,
\end{equation}
\begin{equation}
\label{eqn_Z1}
\begin{split}
\left.\left\langle\nabla Z_1,V\right\rangle\right|_{\{Z_1=1\}}&=2Z_1(X_1-X_2)\leq 2Z_1\left(\sqrt{\frac{Z_2}{Z_1}}-\sqrt{Z_1Z_2}\right)=0.
\end{split}
\end{equation}
Therefore, the vector field $V$ restricted on these faces points inward. For the face $\mathcal{F}\cap \left\{2(\sqrt{Z_2}-\sqrt{Z_3})+X_3-X_2=0\right\},$ we have 
\begin{equation}
\label{eqn_Z2 Z3 X2 X3}
\left. \left\langle \nabla \left(2(\sqrt{Z_2}-\sqrt{Z_3})+X_3-X_2\right),V \right\rangle \right|_{\mathcal{F}\cap \left\{2(\sqrt{Z_2}-\sqrt{Z_3})+X_3-X_2=0\right\}}=(\sqrt{Z_2}-\sqrt{Z_3})K,
\end{equation}
where
$$K:=1+(4m-4)\sqrt{Z_3}-4\sqrt{Z_2}+2Z_1(\sqrt{Z_2}+\sqrt{Z_3}).$$
The non-negativity of the factor $K$ on $\mathcal{RS}\cap \{1-Z_1\geq 0\}\cap \{X_3\geq 0\}\cap \{Z_2-Z_3\geq 0\}$ is established in \cite[Proposition 3.2]{chi_Non-Shrinking_2024}. Finally, by Proposition \ref{prop_F2}, we have $\left.\langle \nabla F_2,  V\rangle \right|_{\mathcal{F}\cap \{F_2=0\}}\geq 0$. Therefore, the set $\mathcal{F}$ is compact and invariant.
\end{proof}

In the following proposition, we show that the level sets $\{F_l=0\}$ serve as barriers within the 2-dimensional invariant sets of $\mathcal{RF}\cap\mathcal{RS}_{m=0}$ and $\mathcal{RF}\cap\mathcal{RS}_{\text{FS}}$, providing insight into the behavior of the integral curves that represent Ricci-flat metrics.a
\begin{proposition}
\label{prop_U_pm}
The sets
\begin{equation}
\begin{split}
\mathcal{U}_+:&=\mathcal{RF}\cap\mathcal{RS}_{m=0}\cap \{F_2\geq 0\}\cap \{X_2-\sqrt{Z_1Z_2}\geq 0\}\cap \{1-Z_1\geq 0\},\\
\mathcal{U}_-:&=\mathcal{RF}\cap\mathcal{RS}_{m=0}\cap \{F_2\leq 0\}\cap \{X_2-\sqrt{Z_1Z_2}\leq 0\},\\
\mathcal{D}_+:&=\mathcal{RF}\cap\mathcal{RS}_{\text{FS}}\cap \{F_{2m+2}\geq 0\}\cap \{X_2-\sqrt{Z_1Z_2}\geq 0\}\cap \{1-Z_1\geq 0\},\\
\mathcal{D}_-:&=\mathcal{RF}\cap\mathcal{RS}_{\text{FS}}\cap \{F_{2m+2}\leq 0\}\cap \{X_2-\sqrt{Z_1Z_2}\leq 0\}
\end{split}
\end{equation}
are invariant.
The integral curve $\xi(2,\pi,0,0)$ is the invariant set $\mathcal{U}_0:=\mathcal{U}_+\cap \mathcal{U}_-$.
The integral curve $\xi(2m+2,0,0,0)$ is the invariant set $\mathcal{D}_0:=\mathcal{D}_+\cap \mathcal{D}_-$.
\end{proposition}
\begin{proof}
Since $H=1$ and $X_3=0$ in $\mathcal{RF}\cap \mathcal{RS}_{m=0}$, the equality $X_1=1-2X_2$ holds. By \eqref{eqn_derivative of Fl}, we have
\begin{equation}
\label{eqn_barrier F2}
\begin{split}
\left.\langle \nabla F_2,  V\rangle \right|_{\mathcal{U}_\pm\cap \{F_2=0\}}&=\left(2\sqrt{\frac{Z_2}{Z_1}}(1-X_1)-4 Z_2\right)(1-Z_1)\\
&= 4(X_2- \sqrt{Z_1Z_2})\sqrt{\frac{Z_2}{Z_1}}(1-Z_1).
\end{split}
\end{equation}
On $\mathcal{U}_\pm\cap \{X_2-\sqrt{Z_1Z_2}=0\}$, we have $F_2=3X_2-1+2\left(\sqrt{\frac{Z_2}{Z_1}}-\sqrt{Z_1Z_2}\right)=\sqrt{Z_1Z_2}-1+2\sqrt{\frac{Z_2}{Z_1}}$. Therefore,
\begin{equation}
\label{eqn_barrier X_2-sqrtZ1Z2}
\begin{split}
\left.\langle \nabla (X_2-\sqrt{Z_1Z_2}),  V\rangle \right|_{\mathcal{U}_\pm\cap \{X_2-\sqrt{Z_1Z_2}=0\}}&= -X_2+R_2-\sqrt{Z_1Z_2}(X_1-2X_2)\\
&=-\sqrt{Z_1Z_2}+4Z_2-2Z_1Z_2-\sqrt{Z_1Z_2}(1-4\sqrt{Z_1Z_2})\\
&=2\sqrt{Z_1Z_2}F_2.
\end{split}
\end{equation}
Since $F_2\geq 0$, the inequality \eqref{eqn_Z1} still holds in $\mathcal{U}_+$ by \eqref{eqn_barrier F2} and \eqref{eqn_barrier X_2-sqrtZ1Z2}. Therefore, the set $\mathcal{U}_+$ is invariant. By \eqref{eqn_Z1} and \eqref{eqn_barrier F2}, the set $\mathcal{U}_0$ is a 1-dimensional invariant set. By \eqref{eqn_genearllinarized_Pcp}, the integral curve $\xi(2,\pi,0,0)$ is tangent to $\mathcal{U}_0$. Therefore, the integral curve is $\mathcal{U}_0$. By \eqref{eqn_barrier X_2-sqrtZ1Z2}, it is clear that $\mathcal{U}_-\cap \{1-Z_1\leq 0\}$ is invariant. If an integral curve in $\mathcal{U}_-$ escapes, it has to transversally intersect $\mathcal{U}_0$, a contradiction. The set $\mathcal{U}_-$ is thus invariant.

The statements on $\mathcal{D}_\pm$ are proved with similar computations. Since $H=1$ and $X_2-X_3=0$ in $\mathcal{RF}\cap \mathcal{RS}_{\text{FS}}$, the equality $X_1=1-(4m+2)X_2$ holds. With $Z_2-Z_3=0$, we have analogous equations to \eqref{eqn_barrier F2} and \eqref{eqn_barrier X_2-sqrtZ1Z2}. The integral curve $\xi(2m+2,0,0,0)$ is thus identified as an algebraic curve $\mathcal{D}_0$.
\end{proof}

\begin{lemma}
\label{lem_long_existing xi_1}
Integral curves in $\left\{\xi(k,\theta,s_4,s_5)\mid k\in\{1,2\}, \theta\in [0,\pi], s_4,s_5\geq 0\right\}$ are defined on $\mathbb{R}$. Metrics represented by these integral curves with $\theta<\pi$ are complete.
\end{lemma}
\begin{proof}
At $p_0$, the inequalities $1-Z_1>0$ and $X_1>0$ hold. The other defining inequalities in \eqref{eqn_setF} vanish at $p_0$. By \eqref{eqn_genearllinarized_Pcp} and \eqref{eqn_s1s2s3 in theta}, the defining inequalities
$$Z_2-Z_3\geq 0,\quad 2(\sqrt{Z_2}-\sqrt{Z_3})+X_3-X_2\geq 0,\quad X_2\geq 0,\quad X_3\geq 0$$
hold near $p_0$. 

We focus on the sign of $F_2$ near $p_0$. By the linearized solution, we have $\lim\limits_{\eta\to\infty} F_2=-1+\frac{2}{k}$. If $k=1$, the function is positive initially. Consider $k=2$. By \eqref{eqn_derivative of Fl}, we have 
\begin{equation}
\label{eqn_derivative of F2}
\begin{split}
F_2'&=F_2\left(G-\frac{\epsilon}{2}W-1+4\sqrt{Z_1Z_2}\right)\\
&\quad +(1-X_1-2\sqrt{Z_1Z_2})2\sqrt{\frac{Z_2}{Z_1}}(1-Z_1)+4mZ_3(1-Z_1).
\end{split}
\end{equation}
By \eqref{eqn_genearllinarized_Pcp}, we also have 
\begin{equation}
\label{eqn_factors with sign}
\begin{split}
&G-\frac{\epsilon}{2}W-1+4\sqrt{Z_1Z_2}\\
&\sim      \left(-(16m^2+24m)s_1-(8m^2+16m)\sqrt{2}s_3-2s_4-(8m+6)\frac{\epsilon}{2}s_5\right)e^{2\eta}+O(e^{(2+\delta)\eta}),\\
&1-X_1-2\sqrt{Z_1Z_2}\\
&\sim \left((8m^2+12m)s_1+(4m^2+8m)\sqrt{2}s_3+s_4+(4m+2)\frac{\epsilon}{2}s_5\right)e^{2\eta}+O(e^{(2+\delta)\eta}).
\end{split}
\end{equation}
If $F_2<0$ initially, the derivative $F_2'$ is initially non-negative by \eqref{eqn_factors with sign}, a contradiction. Note that leading terms in \eqref{eqn_factors with sign} vanish simultaneously only if $(k,\theta,s_4,s_5)=(2,\pi,0,0)$. By Proposition \ref{prop_U_pm}, the integral curve $\xi(2,\pi,0,0)$ still stays in $\mathcal{F}$. Therefore, each integral curve in $\left\{\xi(k,\theta,s_4,s_5)\mid k\in\{1,2\}, \theta\in [0,\pi], s_4,s_5\geq 0\right\}$ is in the set $\mathcal{F}$, with $\xi(2,\pi,0,0)$ being the only one that lies on the boundary $\partial\mathcal{F}$. By the escape lemma, each integral curve in the family is defined on $\mathbb{R}$.
\end{proof}

%

\subsection{$k\geq 3$}
Since $\lim\limits_{\eta\to-\infty} F_2=-1+\frac{2}{k}$, the function $F_2$ is initially \emph{negative} along $\xi(k,\theta,s_4,s_5)$ if $k\geq 3$. Therefore, the invariant set $\mathcal{F}$ used for $k\in\{1,2\}$ is no longer applicable. We hence apply an alternative method to analyze this case.

Define the following sets.
\begin{equation}
\begin{split}
&\mathcal{A}:=\mathcal{RS}\cap \{1-Z_1\geq 0\} \cap \{X_1-X_2\leq 0\} \cap \{Z_2-Z_3\geq 0\}\\
&\quad \cap \left\{2(\sqrt{Z_2}-\sqrt{Z_3})+X_3-X_2\geq 0\right\}\cap \{X_1,X_2,X_3\geq 0\},\\
&\mathcal{B}:=\mathcal{RS}\cap \{1-Z_1\geq 0\}\cap \{X_1-X_2\geq 0\}\cap \{Z_2-Z_3\geq 0\}\\
&\quad \cap \left\{2(\sqrt{Z_2}-\sqrt{Z_3})+X_3-X_2\geq 0\right\}\cap \{X_1,X_2,X_3\geq 0\},\\
&\mathcal{C}:=\mathcal{RS}\cap \{1-Z_1\leq 0\} \cap \{X_1-X_2\geq 0\}.
\end{split}
\end{equation}
The compact invariant set $\mathcal{A}$ plays an important role in our previous construction; see \cite[Lemma 3.3]{chi_Non-Shrinking_2024}. Unlike the previous setting, the critical point $p_0$ does not belong to $\mathcal{A}$. In the cases $k\in\{1,2\}$, this difficulty is overcome by the invariant set $\mathcal{F}$, which contains the trajectories near $p_0$ and guarantees that they enter $\mathcal{A}$ eventually. For $k\geq 3$, however, integral curves that are not in $\mathcal{F}$ initially. We hence introduce the following trichotomy to describe possible behaviors of integral curves emanating $p_0$.

\begin{lemma}
\label{lem_B, C or D}
Sets $\mathcal{A}$ and $\mathcal{B}$ are compact. Sets $\mathcal{A}$ and $\mathcal{C}$ are invariant.
Each $\xi(k,\theta,s_4,s_5)$ stays in $\mathcal{B}$, or eventually enters $\mathcal{C}$ through $\mathcal{B}\cap \{1-Z_1=0\}\cap \{X_1-X_2>0\}$, or eventually enters $\mathcal{A}$ through $\mathcal{B}\cap \{1-Z_1>0\}\cap \{X_1-X_2=0\}$.
\end{lemma}
\begin{proof}
By the inequalities $1-Z_1\geq 0$ and $Z_2-Z_3\geq 0$, it is clear that $\mathcal{A}$ and $\mathcal{B}$ are compact. Since
\begin{equation}
\label{eqn_D set boundary}
\begin{bmatrix}
Z_1\\
X_1-X_2
\end{bmatrix}'=\begin{bmatrix}
2Z_1(X_1-X_2)\\
(X_1-X_2)\left(G-\frac{\epsilon}{2}W-1\right)+(4 Z_2+4mZ_3)(Z_1-1)
\end{bmatrix},
\end{equation}
the set $\mathcal{C}$ is invariant.

By \eqref{eqn_genearllinarized_Pcp_2}, each $\xi(k,\theta,s_4,s_5)$ is initially in $\mathcal{B}$. The integral curve is in the boundary $\partial\mathcal{B}$ only if $\theta=0$. By \eqref{eqn_Z2 Z3}-\eqref{eqn_Z2 Z3 X2 X3} and the fact that $R_i\geq 0$ in $\mathcal{B}$, if the integral escapes $\mathcal{B}$, it does not leave the set through the following faces:
$$
\mathcal{B} \cap \{Z_2-Z_3=0\},\quad \mathcal{B}\cap \left\{2(\sqrt{Z_2}-\sqrt{Z_3})+X_3-X_2=0\right\},\quad \mathcal{B}\cap \{X_i=0\}.
$$
In other words, if a $\xi(k, \theta, s_4,s_5)$ leaves $\mathcal{B}$ at some $\eta_*$, either $1-Z_1$ or $X_1-X_2$ vanishes at $\eta_*$. These two functions cannot vanish simultaneously as the set $\mathcal{RS}_{\text{round}}$ is invariant while $p_0 \notin \mathcal{RS}_{\text{round}}$. If the integral curve leaves $\mathcal{B}$ through a point with $1-Z_1=0$ and $X_1-X_2>0$, it enters the invariant set $\mathcal{C}$. If the integral curve leaves $\mathcal{B}$ through a point with $1-Z_1>0$ and $X_1-X_2=0$, then it enters the invariant set $\mathcal{A}$. The proof is complete.
\end{proof}

Note that $\mathcal{F}\subset\mathcal{B}$. While $\mathcal{B}$ being a larger set that include $\xi(k, \theta, s_4,s_5)$ with any $k\geq 1$, the trade-off is losing the invariance of the set. By Lemma \ref{lem_B, C or D}, it is necessary to determine the parameters for which the integral curve enters $\mathcal{A}$ or stays in $\mathcal{B}$. In the Ricci-flat case with $\theta\in\{0,\pi\}$, this question is straightforward to resolve, as demonstrated in the following proposition.

\begin{proposition}
\label{prop_where do xi go}
The integral curve $\xi(k,\pi,0,0)$ eventually enters $\mathcal{A}$ if $k=1$, eventually enters $\mathcal{C}$ if $k\geq 3$. The integral curve $\xi(k,0,0,0)$ eventually enters $\mathcal{A}$ if $k\in [1,2m+1]$, eventually enters $\mathcal{C}$ if $k\geq 2m+3$. 
\end{proposition}
\begin{proof}
Let $X:=X_1-X_2$. The equalities $X_1=\frac{1+2X}{3}$ and $X_2=\frac{1-X}{3}$ hold on $\mathcal{E}\cap \mathcal{RS}_{m=0}$. Furthermore, the system \eqref{eqn_Polynomial_soliton_equation} restricted on the $2$-dimensional invariant set $\mathcal{RF}\cap \mathcal{RS}_{m=0}$ is equivalent to
\begin{equation}
\label{eqn_RF system}
\begin{bmatrix}
X\\
Z_1
\end{bmatrix}'=\begin{bmatrix}
X (G-1)+4Z_2(Z_1-1)\\
2Z_1 X
\end{bmatrix}
\end{equation}
where $Z_2=\frac{1-G}{8-2Z_1}\geq 0$. By Proposition \ref{prop_U_pm}, the integral curve $\xi(2,\pi,0,0)$ joins the critical point $(1,0)$ and $(0,1)$. As the critical point $(0,1)$ is a saddle in \eqref{eqn_RF system}, each $\xi(k,\pi,0,0)$ does not converge to this critical point if $k\neq 2$.

The integral curve $\xi(1,\pi,0,0)$ is initially in $\mathcal{B}\cap \mathcal{U}_+$. If $\xi(1,\pi,0,0)$ escapes $\mathcal{B}$ through points in $\mathcal{B}\cap \{1-Z_1=0\}\cap \{X_1-X_2>0\}$, the inequality $F_2<0$ holds at the intersection point, meaning the integral curve escapes $\mathcal{U}_+$ first. This is a contradiction to Proposition \ref{prop_U_pm}. If $\xi(1,\pi,0,0)$ stays in $\mathcal{B}\cap \mathcal{RF}\cap \mathcal{RS}_{m=0}$, the monotonicity of $Z_1$ forces the integral curve to converge to $(0,1)$, a contradiction. Therefore, the integral curve $\xi(1,\pi,0,0)$ eventually enters $\mathcal{A}$.

On the other hand, the integral curve $\xi(k,\pi,0,0)$ is initially in  $\mathcal{B}\cap \mathcal{U}_-$ if $k\geq 3$. If the integral curve escapes $\mathcal{B}$ through points in $\mathcal{B}\cap \{1-Z_1>0\}\cap \{X_1-X_2=0\}$, the inequality $F_2>0$ holds at the intersection point, meaning the integral curve escapes $\mathcal{U}_-$ first. This is a contradiction to Proposition \ref{prop_U_pm}. If $\xi(k,\pi,0,0)$ stays in $\mathcal{B}\cap \mathcal{RF}\cap \mathcal{RS}_{m=0}$, the integral curve converges to $(0,1)$, a contradiction. Therefore, the integral curve $\xi(k,\pi,0,0)$ eventually enters $\mathcal{C}$ if $k\geq 3$.

By Proposition \ref{prop_U_pm}, the $2$-dimensional set $\mathcal{B}\cap \mathcal{RF}\cap \mathcal{RS}_{\text{FS}}$ is divided into two components by the integral curve $\xi(2m+2,0,0,0)$. The statement for $\xi(k,0,0,0)$ is proved with the similar argument.
\end{proof}

\begin{proposition}
\label{prop_concave potential}
The function $\frac{H-1}{\sqrt{-Q}}$ monotonically decreases along each $\xi(k, \theta, s_4,s_5)$ with $s_4>0$ as long as the integral curve is in $\mathcal{B}$.
\end{proposition}
\begin{proof}
We first claim 
\begin{equation}
\label{eqn_negative quantity}
Q+1-H+\frac{\epsilon}{2}W(1-H)<0
\end{equation}
 along each $\xi(k, \theta, s_4,s_5)$ in $\mathcal{B}$ with $s_4>0$. By \eqref{eqn_linearized Q-H+1}, we have $Q+1-H+\frac{\epsilon}{2}W(1-H)\sim -s_4e^{\eta}$. Hence, the function is initially negative if $s_4>0$. By \eqref{eqn_characteriaztion_of_non_trivial_soliton}, we have 
\begin{equation}
\label{eqn_interesting function}
\begin{split}
&\left(Q+1-H+\frac{\epsilon}{2}W(1-H)\right)'\\
&= 2Q\left(G-\frac{\epsilon}{2}W\right)+\epsilon(H-1)W+2\frac{\epsilon}{2}W\left(G-\frac{\epsilon}{2}W\right)(1-H)\\
&\quad -\left(\frac{\epsilon}{2}W+1\right)\left((H-1)\left(G-\frac{\epsilon}{2}W-1\right)+Q\right)\\
&= \left(Q+1-H+\frac{\epsilon}{2}W(1-H)\right)\left(2G-\frac{3\epsilon}{2}W-1\right) +(1-H)\left(\frac{\epsilon}{2}W-1\right)G.\\
\end{split}
\end{equation}
For $\epsilon=0$, the last term is non-positive in $\mathcal{RS}$. For $\epsilon=1$, note that $R_s\geq 0$ in $\mathcal{B}$ and hence $(n-1)\frac{\epsilon}{2}W=Q+1-G-R_s\leq 1$. Therefore, the last term is still non-positive. The inequality $Q+1-H+\frac{\epsilon}{2}W(1-H)< 0$ holds.

We have
\begin{equation}
\begin{split}
\left(\frac{H-1}{\sqrt{-Q}}\right)'&=\frac{-Q(Q+1-H)+\frac{\epsilon}{2}W(1-H)^2}{\sqrt{-Q}(-Q)}\\
&\leq \frac{-Q(Q+1-H)+(H-1-Q)(1-H)}{\sqrt{-Q}(-Q)}\quad \text{by \eqref{eqn_negative quantity}}\\
&=\frac{(Q+1-H)^2}{Q\sqrt{-Q}}\\
&\leq 0.
\end{split}
\end{equation}
The proof is complete.
\end{proof}

\begin{proposition}
\label{prop_bounded 1-H and Q}
Along each $\xi(k, \theta, s_4,s_5)$ with $s_4>0$ while the integral curve is in $\mathcal{B}$, there exists an $\eta_*$ and a $C_1>0$ small enough such that $\sqrt{-Q}\geq C_1$ and $1-H\geq C_1^2$ for $\eta>\eta_*$.
\end{proposition}
\begin{proof}
By Proposition \ref{prop_concave potential}, there exists an small enough $C_1>0$ and an $\eta_*$ such that $\frac{1-H}{\sqrt{-Q}}\geq C_1$ for $\eta>\eta_*$, where $C_1=\left(\frac{1-H}{\sqrt{-Q}}\right)(\eta_*)>0$. By \eqref{eqn_negative quantity}, we have
\begin{equation}
\begin{split}
0&\geq -\frac{\epsilon}{2}W(1-H)>Q+1-H\geq \sqrt{-Q}(C_1-\sqrt{-Q})
\end{split}
\end{equation}
for $\eta>\eta_*$ while the integral curve is in $\mathcal{B}$.
Therefore, the inequalities $\sqrt{-Q}\geq C_1$ and $1-H\geq C_1^2$ hold.
\end{proof}

We proceed to prove the following lemma, which is a generalization to \cite[Lemma 5.4]{appleton_family_2017}.
\begin{lemma}
\label{lem_long exiting xik}
For each fixed $k\geq 3$ and $\theta\in\left[0,\pi\right]$, there exists an $\tilde{\alpha}_{k,\theta}\geq 0$ such that each $\xi(k, \theta, s_4,s_5)$ with $s_5\geq 0$ stays in $\mathcal{B}$ or eventually enters $\mathcal{A}$ if $s_4\geq \tilde{\alpha}_{k,\theta}$. Each $\xi(k, \theta, s_4,s_5)$ with $s_4\geq \tilde{\alpha}_{k,\theta}$ and $s_5\geq 0$ is thus defined on $\mathbb{R}$.
\end{lemma}
\begin{proof}
Initially the integral curve $\xi(k, \theta, s_4,s_5)$ is in $\mathcal{B}$. For $s_4>0$, the function $Q$ is initially negative. Furthermore, as $R_i\geq 0$ in $\mathcal{B}$, we have $-Q\in (0,1)$ along $\xi(k, \theta, s_4,s_5)$ in $\mathcal{B}$. We have
\begin{equation}
\label{eqn_second derivative of Z1/Q}
\begin{split}
\left(\frac{\sqrt{Z_1}'}{\sqrt{-Q}}\right)'&= \left(\frac{\sqrt{Z_1}'}{\sqrt{-Q}}\right)(X_1-X_2-1)+\frac{R_1-R_2}{\sqrt{-Q}}+\frac{\frac{\epsilon}{2}W(1-H)\sqrt{Z_1}'}{Q\sqrt{-Q}}\\
&\leq \left(\frac{\sqrt{Z_1}'}{\sqrt{-Q}}\right)(X_1-X_2-1)\\
&\leq \left(\frac{\sqrt{Z_1}'}{\sqrt{-Q}}\right)(H-1)\leq 0.
\end{split}
\end{equation}
The function $\frac{\sqrt{Z_1}'}{\sqrt{-Q}}$ monotonically decreases while $\xi(k, \theta, s_4,s_5)$ is in $\mathcal{B}$. 

By Proposition \ref{prop_bounded 1-H and Q}, there exists a small enough $\eta_*$ such that 
\begin{equation}
\begin{split}
\left(\frac{\sqrt{Z_1}'}{\sqrt{-Q}}\right)'&\leq -\left(\frac{\sqrt{Z_1}'}{\sqrt{-Q}}\right)C_1^2,\quad C_1=\left(\frac{1-H}{\sqrt{-Q}}\right)(\eta_*)
\end{split}
\end{equation}
holds along $\xi(k, \theta, s_4,s_5)$ in $\mathcal{B}$ for $\eta>\eta_*$. We have
\begin{equation}
\frac{\sqrt{Z_1}'}{\sqrt{-Q}}\leq C_2 e^{-C_1^2(\eta-\eta_*)},\quad C_2=\left(\frac{\sqrt{Z_1}'}{\sqrt{-Q}}\right)(\eta_*)
\end{equation}
By \eqref{eqn_genearllinarized_Pcp} and the monotonicity of $\frac{\sqrt{Z_1}'}{\sqrt{-Q}}$, we have 
\begin{equation}
\begin{split}
C_2&\leq \lim\limits_{\eta\to-\infty}\left(\frac{\sqrt{Z_1}'}{\sqrt{-Q}}\right)(\xi(k, \theta, s_4,s_5))\\
&=\lim\limits_{\eta\to-\infty}\left(\frac{\sqrt{Z_1}(X_1-X_2)}{\sqrt{-Q}}\right)(\xi(k, \theta, s_4,s_5))=\frac{(\sqrt{s_1}+\sqrt{s_2})k}{\sqrt{2s_4}}.
\end{split}
\end{equation}
By \eqref{eqn_s1s2s3 in theta}, we have $C_2\leq \frac{k}{\sqrt{s_4}}$.
It follows that
\begin{equation}
\label{eqn_Z1'}
\begin{split}
\sqrt{Z_1}'&\leq \frac{k}{\sqrt{s_4}}\sqrt{-Q} e^{-C_1^2(\eta-\eta_*)}\leq \frac{k}{\sqrt{s_4}} e^{-C_1^2(\eta-\eta_*)}
\end{split}
\end{equation}

Integrating \eqref{eqn_Z1'} for both sides, we obtain
\begin{equation}
\begin{split}
\sqrt{Z_1}&\leq \frac{k}{\sqrt{s_4}} \frac{1-e^{-C_1^2(\eta-\eta_*)}}{C_1^2}+(\sqrt{Z_1})(\eta_*)\leq \frac{k}{\sqrt{s_4}} \frac{1}{C_1^2}+(\sqrt{Z_1})(\eta_*)
\end{split}
\end{equation}
for $\eta>\eta_*$ while $\xi(k, \theta, s_4,s_5)$ is in $\mathcal{B}$. By \eqref{eqn_Q and 1-H}, we have
$$C_1^2=\left(\frac{(1-H)^2}{-Q}\right)(\eta_*)=\frac{s_4}{2}e^{2\eta_*}+O(e^{(2+\delta)\eta_*}).$$
Choose a small enough $\eta_*$ and a large enough $s_4$, the inequality $\sqrt{Z}_1\leq 1$ holds along $\xi(k, \theta, s_4,s_5)$ while the integral curve is in $\mathcal{B}$. Therefore, the integral curve stays in $\mathcal{B}$, or eventually enters $\mathcal{A}$ by Lemma \ref{lem_B, C or D}. The existence of $\tilde{\alpha}_{k,\theta}$ is established.
\end{proof}

For each $k\geq 3$ and $\theta\in [0,\pi]$, define
\begin{align}
\alpha_{k,\theta}&=\inf\left\{\tilde{\alpha}_{k,\theta}\mid \text{$\xi(k,\theta,s_4,0)$ does not enter $\mathcal{C}$ for any $s_4\geq \tilde{\alpha}_{k,\theta}$}\right\},\\
\beta_{k,\theta}&=\inf\left\{\tilde{\alpha}_{k,\theta}\mid \text{$\xi(k,\theta,s_4,s_5)$ does not enter $\mathcal{C}$ for any $s_4\geq \tilde{\alpha}_{k,\theta}$ and $s_5\geq 0$ }\right\}.
\end{align}
By Lemma \ref{lem_long exiting xik}, such $\alpha_{k,\theta}$ and $\beta_{k,\theta}$ exist. It is clear that $\beta_{k,\theta}\geq \alpha_{k,\theta}$. 
\begin{lemma}
\label{lem_for thm 1.2 1.3}
For each fixed $k\geq 3$ and $\theta\in\left[0,\pi\right]$, the integral curve $\xi(k,\theta,\beta_{k,\theta},s_5)$ does not enters $\mathcal{C}$ for any $s_5\geq 0$. The integral curve  $\xi(k,\theta,\alpha_{k,\theta},0)$ stays in $\mathcal{B}$ if $\alpha_{k,\theta}>0$.
\end{lemma}
\begin{proof}
We first prove the statement for $\beta_{k,\theta}$. Suppose otherwise, there exists ${s_{5}}_*\geq 0$ such that $\xi(k,\theta,\beta_{k,\theta},{s_{5}}_*)$ eventually enters $\mathcal{C}$. By the the continuous dependence on parameters, there exists a small $\delta>0$ such that $\xi(k,\theta,\beta_{k,\theta}+\delta,s_{5*})$ eventually enters $\mathcal{C}$, a contradiction to the definition of $\beta_{k,\theta}$.

Assume $\alpha_{k,\theta}>0$. The argument for $\beta_{k,\theta}$ carries over so that $\xi(k,\theta,\alpha_{k,\theta},0)$ does not enters $\mathcal{C}$. If $\xi(k,\theta,\alpha_{k,\theta},0)$ enters $\mathcal{A}$, there exists a small $\delta>0$ such that $\alpha_{k,\theta}-\delta>0$ and $\xi(k,\theta,s_4,0)$ does not enter $\mathcal{C}$ and enters $\mathcal{A}$ if $s_4\in (\alpha_{k,\theta}-\delta,\alpha_{k,\theta}]$, a contradiction to the definition of $\alpha_{k,\theta}$.
\end{proof}

\begin{lemma}
\label{lem_for thm 1.4}
If $k\geq 3$, there are infinitely many $\theta \in (0,\pi)$ such that $\xi(k, \theta, s_4 ,0)$ stays in $\mathcal{B}$ for some $s_4>0$.
If $k\in[3,2m+1]$, there exists some $\theta_\star \in (0,\pi)$ such that $\xi(k, \theta_\star, 0,0)$ stays in $\mathcal{B}$ and is defined on $\mathbb{R}$.
\end{lemma}
\begin{proof}
Consider $k\geq 3$. Define $\alpha:=\max\limits_{\theta\in [0,\pi]}\{\alpha_{k,\theta}\}$. Each integral curve in the family
$$\left\{\xi(k,\theta,s_4,0)\mid \theta\in [0,\pi], s_4\geq \alpha\right\}$$ enters $\mathcal{A}$ or stays in $\mathcal{B}$. By Proposition \ref{prop_where do xi go}, the integral curve $\xi(k,\pi,0,0)$ enters $\mathcal{C}$, and thus $\alpha\geq \alpha_{k,0}>0$. Suppose $\xi(k,\theta,\alpha,0)$ enters $\mathcal{A}$. Consider the continuous 1-parameter family 
$$\left\{\xi\left(k,(1-\tau)\pi+\tau\theta,\tau\alpha,0\right) \mid \tau\in [0,1]\right\}.$$
There exists some $\tau_\theta \in (0,1)$ such that $\xi_\star(k,\theta):=\xi\left(k,(1-\tau_\theta)\pi+\tau_\theta \theta,\tau_\theta \alpha,0\right)$ stays in $\mathcal{B}$ by the continuous dependence on parameter. Since $\xi_\star(k,\theta_1)\neq \xi_\star(k,\theta_2)$ if $\theta_1\neq \theta_2$, there are infinitely many such $\xi_\star(k,\theta)$.

If $k\in[3,2m+1]$, the integral curve $\xi(k,0,0,0)$ enters $\mathcal{A}$, and $\xi(k,\pi,0,0)$ enters $\mathcal{C}$ by Proposition \ref{prop_where do xi go}. Consider the continuous 1-parameter family $\left\{\xi\left(k,\theta,0,0\right) \mid \theta\in [0,\pi]\right\}.$ There exists some $\theta_\star \in (0,\pi)$ such that $\xi(k,\theta_\star,0,0)$ stays in $\mathcal{B}$ by the continuous dependence on parameter. The integral curve is defined on $\mathbb{R}$ by the compactness of $\mathcal{B}$.
\end{proof}

\section{Asymptotics}
In general, determining the asymptotic geometry of complete cohomogeneity one Ricci solitons consists of two steps. The first is to identify the limiting critical point of the corresponding integral curve. The second is to extract the asymptotic behavior of the metric from this limiting information. Much of the asymptotic analysis developed in previous works extends to the present setting. We summarize the resulting classification below and provide a brief proof.
\begin{proposition}[Limit Classification]
\label{prop_Limit Classification}
For an integral curve $\xi(k,\theta,s_4,s_5)$ with $s_4,s_5\geq 0$ that remains in $\mathcal{A}\cup\mathcal{B}$, the integral curve converges to:
\begin{enumerate}
\item
$q_0=(0,0,0,\mu^2,0,0,0,0)$ for some $\mu\in [0,1]$ if $s_4>0$. 
 \item
$q_1=\left(0,\frac{1}{n-1},\frac{1}{n-1},0,\frac{1}{(n-1)^2}\frac{n-2}{n+1},\frac{1}{(n-1)^2}\frac{n-2}{n+1},\frac{1}{(n-1)^2}\frac{n-2}{n+1},0\right)$
if the integral curve enters $\mathcal{A}$ and $(\theta,s_4,s_5)=(0,0,0)$;
\item
$q_2=\left(0,\frac{1}{n-1},\frac{1}{n-1},0,(m+1)^2z_*,z_*,(m+1)z_*,0\right)$, where $ z_*=\frac{1}{(n-1)^2}\frac{n-2}{4(m^2+3m+1)}$, if the integral curve enters $\mathcal{A}$ and $\theta\in(0,\pi)$, $(s_4,s_5)=(0,0)$;
\item
$p_1=\left(\frac{1}{n},\frac{1}{n},\frac{1}{n},1,\frac{1}{n^2},\frac{1}{n^2},\frac{1}{n^2},0\right)$ or $p_2=\left(\frac{1}{n},\frac{1}{n},\frac{1}{n},1,(2m+3)^2z_{*},z_{*},(2m+3)z_{*},0\right), z_{*}=\frac{n-1}{n^2}\frac{1}{2(2m+3)^2+4m}$
 if the integral curve remains in $\mathcal{B}$ and $(s_4,s_5)=(0,0)$;
\item
the invariant set 
$P=\left\{\left(\frac{1}{n},\frac{1}{n},\frac{1}{n},z_1,0,0,0,\frac{2}{n\epsilon}\right)\mid z_1\in[0,1] \right\}$ if $s_4=0$ and $s_5>0$.
\end{enumerate}
\end{proposition}
\begin{proof}
If $s_4>0$ and $s_5=0$, the function $Q$ is initially negative and we have $Q'=QG<0$. Since $\mathcal{A}\cup\mathcal{B}$ is compact, the function $Q$ converges. Necessarily, each $X_i$ vanishes at the $\omega$-limit set. Since the $\omega$-limit set is invariant, all coordinates except $Z_1$ vanish at the set. For the non-Einstein expanding case, we have $\lim\limits_{\eta\to\infty}W=0$ by \cite[Proposition 1.3 and Proposition 1.18]{buzano_family_2015}. Then case (1) follows from \cite[pp. 1112-1115]{dancer_non-kahler_2009}.

For the Ricci-flat case ($s_4,s_5=0$), the integral curves are on the invariant set $\mathcal{E}$. Asymptotic analysis in \cite[Section 5]{chi2021einstein} carries over for integral curves that enter $\mathcal{A}$, proving cases (2), (3). The argument in fact becomes easier for integral curves that remain $\mathcal{B}$ since $Z_1$ does not converge to zero. Specifically, we have $nG\geq 1$ by the Cauchy--Schwarz inequality and hence the monotonicity of the B\"ohm's functional
$$
\left(\frac{Z_2^{2m+3}Z_3^{2m}}{Z_1}\right)'=\frac{Z_2^{2m+3}Z_3^{2m}}{Z_1}(nG-1)\geq 0.
$$
Thus, the $\omega$-limit set is contained in $\mathcal{E}\cap \{nG=1\}$. By the invariance of the $\omega$-limit set, the set is either $\{p_1\}$ or $\{p_2\}$. We thus prove the case $(4)$.

Finally, for the negative Einstein case ($s_4=0,s_5>0$), since $G+R_s+(4m+2)\frac{\epsilon}{2} W=1$ and $R_s\geq 0$ in $\mathcal{E}\cap \{\mathcal{A}\cup\mathcal{B}\}$, we have $\frac{1}{4m+3}\geq \frac{\epsilon}{2}W$ the Cauchy--Schwarz inequality. Hence $G-\frac{\epsilon}{2}W\geq 0$ and $W$ is monotonic increasing. The $\omega$-limit set is contained in $\mathcal{E}\cap \{\mathcal{A}\cup\mathcal{B}\}\cap \{G-\frac{\epsilon}{2}W=0\}$. Since $R_s\geq 0$ in $\mathcal{A}\cup\mathcal{B}$, the $\omega$-limit set is contained in $P$. We hence prove the case $(5)$.
\end{proof}

\subsection{Einstein metrics ($s_4=0$)}
We first consider the Einstein case, where the asymptotic geometry is encoded largely by the limiting critical point. Cases (2)-(3) in Proposition \ref{prop_Limit Classification} is covered by \cite[Section 4]{chi2021einstein}. We recall the results below.
\begin{lemma}
\label{lem_asymp for RF in A}
Each $\xi(k,\theta,0,0)$ that enters $\mathcal{A}$ is ALC. The base of the asymptotic cone is the Fubini--Study $\mathbb{CP}^{2m+1}$ if $\theta=0$; non-K\"ahler $\mathbb{CP}^{2m+1}$ if $\theta\in(0,\pi)$.
\end{lemma}
\begin{proof}
By Proposition \ref{prop_Limit Classification}, the integral curve converges to $q_1$ if $\theta=0$; $q_2$ if $\theta\in (0,\pi)$. Since $\dot{b}=\frac{X_2}{\sqrt{Z_2}}$ and $\dot{c}=\frac{X_3}{\sqrt{Z_4}}$, the limits $b_\infty:=\lim\limits_{t\to\infty} \dot{b}$ and $c_\infty:=\lim\limits_{t\to\infty} \dot{c}$ exist. Thus, the functions $b$ and $c$ exhibit linear growth at the infinity. Furthermore, we have 
$\left(\frac{4}{b^2}+\frac{4mb^2}{c^4}\right)\sim \frac{n-2}{t^2}$
as $t\to\infty$. For completeness, we verify that the $a$ converges to a positive constant, which was not explicitly shown in \cite{chi2021einstein}.

To show $\lim\limits_{t\to\infty}a$ exists, consider
\begin{equation}
\begin{split}
\ddot{\sqrt{Z_1}}=\dot{\sqrt{Z_1}}\left(-3\frac{\dot{b}}{b}-4m \frac{\dot{c}}{c} \right)-\sqrt{Z_1}(1-Z_1)\left(\frac{4}{b^2}+\frac{4mb^2}{c^4}\right).
\end{split}
\end{equation}
Since $\sqrt{Z_1}$ decreases to zero in $\mathcal{A}$, the inequality
\begin{equation}
\begin{split}
\ddot{\sqrt{Z_1}}\leq -\dot{\sqrt{Z_1}}\frac{n(1+\delta)}{t}-\sqrt{Z_1}\frac{(n-2)(1-\delta)}{t^2}
\end{split}
\end{equation}
eventually holds for some small $\delta>0$.

Let $c_1=n(1+\delta)$ and $c_2=(n-2)(1-\delta)$. Integrate both sides, the inequality 
\begin{equation}
\begin{split}
\frac{a}{b}=\sqrt{Z_1}&\leq K_1 t^{\frac{-\left(c_1-1-\sqrt{(c_1-1)^2-4c_2}\right)}{2}}+K_2 t^{\frac{-\left(c_1-1+\sqrt{(c_1-1)^2-4c_2}\right)}{2}}
\end{split}
\end{equation}
eventually holds. Since the indexes respectively converge to $-1$ and $-(n-2)$ as $\delta\to 0^+$, the inequality
\begin{equation}
\label{eqn_estimate for a/b}
\begin{split}
\frac{a}{b}\leq K_3 t^{-\frac{2}{3}}
\end{split}
\end{equation}
eventually holds for some $K_3>0$.

Since 
\begin{equation}
\ddot{a}=\left(-2\frac{\dot{b}}{b}-4m\frac{\dot{c}}{c}\right)\dot{a}+\left(\frac{a}{b}\right)^3\left(\frac{2}{b^2}+4m\frac{b^2}{c^4}\right)b ,
\end{equation}
and $a$ is decreasing in $\mathcal{A}$, the inequality
\begin{equation}
\begin{split}
\ddot{a}&\leq -\dot{a}\frac{(n-1)(1+\delta)}{t}+\left(\frac{a}{b}\right)^3\frac{(n-2)(1+\delta)b_\infty}{t}\\
&\leq -\dot{a}\frac{2(n-1)}{t}+K_3^3\frac{2(n-2)b_\infty}{t^3}
\end{split}
\end{equation}
eventually holds. Multiply both sides by $t^{2(n-1)}$ and integrate, we have 
\begin{equation}
\begin{split}
\dot{a}&\leq \frac{2(n-2)b_\infty}{t^{2(n-1)}}\int_{t_0}^t K_3^3 \tau^{2n-5} d\tau \quad \text{by \eqref{eqn_estimate for a/b}}+\frac{K_4}{t^{2(n-1)}}\\
&\leq \frac{b_\infty K_3^3}{t^2}+\frac{K_5}{t^{2(n-1)}}\\
&\leq \frac{K_6}{t^2}
\end{split}
\end{equation}
eventually for some $K_6>0$.
Therefore, the function $a$ is bounded above along an integral curve in $\mathcal{A}$. The ALC asymptotic is established. 
\end{proof}

We proceed to consider Ricci-flat integral curves that remains in $\mathcal{B}$. The following proposition extends \cite[Proposition 5.5]{chi2021einstein}, which identifies the base of the asymptotic cone in terms of the parameter $\theta$.
\begin{proposition}
\label{prop_repellent in B 1}
For a $\xi(k,\theta,0,0)$ that remains in $\mathcal{B}$ and $\sqrt{\frac{Z_3}{Z_2}}<1$ initially, it is necessary that $\nu^2=\lim\limits_{\eta\to\infty}\sqrt{\frac{Z_3}{Z_2}}<1$.
\end{proposition}
\begin{proof}
Suppose $\nu^2=1$. If $\sqrt{\frac{Z_3}{Z_2}}<1$ initially, the function $X_2-X_3$ is not identically zero along the integral curve. If $X_2-X_3$ changes sign infinitely many times, so does 
$$R_2-R_3=(\sqrt{Z_2}-\sqrt{Z_3})((4-2Z_1)\sqrt{Z_2}-(4m+4+2Z_1)\sqrt{Z_3}).$$ Since $\sqrt{\frac{Z_3}{Z_2}}\leq 1$ and $\lim\limits_{\eta\to\infty}Z_1=\mu^2\in [0,1]$ in $\mathcal{B}$, the function $\sqrt{\frac{Z_3}{Z_2}}$ eventually oscillates near $\frac{4-2\mu^2}{4m+4+2\mu^2}<1$. Hence $\lim\limits_{\eta\to\infty}\sqrt{\frac{Z_3}{Z_2}}\neq 1$, a contradiction. Therefore, the function $X_2-X_3$ eventually has sign. Since $\left(\frac{Z_3}{Z_2}\right)'=2\frac{Z_3}{Z_2}(X_2-X_3)$, and $\frac{Z_3}{Z_2}\leq 1$ in $\mathcal{B}$, it is necessary that $X_2-X_3$ being eventually positive given that $\nu^2=1$.

We have
\begin{equation}
\label{eqn_key derivative 1}
\begin{split}
&\left(2(\sqrt{Z_2}-\sqrt{Z_3})+X_3-X_2\right)'\\
&=\left(2(\sqrt{Z_2}-\sqrt{Z_3})+X_3-X_2\right)G+(X_2-X_3)\left(1-4\sqrt{Z_3}\right)\\
&\quad +(\sqrt{Z_2}-\sqrt{Z_3})(-2X_2 +4(m+1)\sqrt{Z_3} - 4\sqrt{Z_2} +2Z_1(\sqrt{Z_2} +\sqrt{Z_3}))
\end{split}
\end{equation}
Since $2(\sqrt{Z_2}-\sqrt{Z_3})+X_3-X_2\geq 0$ in $\mathcal{B}$, we have
\begin{equation}
\begin{split}
&\left(2(\sqrt{Z_2}-\sqrt{Z_3})+X_3-X_2\right)'\\
&\geq  (X_2-X_3)\left(1-4\sqrt{Z_3}\right)\notag\\
&\quad +(\sqrt{Z_2}-\sqrt{Z_3})(-2X_2 +4(m+1)\sqrt{Z_3} - 4\sqrt{Z_2} +2Z_1(\sqrt{Z_2} +\sqrt{Z_3}))
\end{split}
\end{equation}
By Proposition \ref{prop_Limit Classification}, the integral curve converge to $p_1$ if $\nu^2=1$. The coefficients for $X_2-X_3$ and $\sqrt{Z_2}-\sqrt{Z_3}$ above are positive at $p_1$. With $X_2-X_3$ being eventually positive, the function $2(\sqrt{Z_2}-\sqrt{Z_3})+X_3-X_2$ eventually increases. But the function is non-negative in $\mathcal{B}$ and it vanishes at $p_1$. We reach
a contradiction.
\end{proof}

We have the following lemma.
\begin{lemma}
\label{lem_asymp for RF in B}
A $\xi(k,\theta,0,0)$ that remains in $\mathcal{B}$ is AC. The base of the cone is the standard $\mathbb{S}^{4m+3}/\mathbb{Z}_k$ if $\theta=0$; Jensen $\mathbb{S}^{4m+3}/\mathbb{Z}_k$ if $\theta\in (0,\pi)$.
\end{lemma}
\begin{proof}
By (4) in Proposition \ref{prop_Limit Classification}, the integral curve converges to $p_1$ or $p_2$. The linear growth of all metric components are derived by evaluating $X_1\sqrt{\frac{Z_1}{Z_2}}$, $\frac{X_2}{\sqrt{Z_2}}$, and $\frac{X_3}{\sqrt{Z_4}}$ at $p_1$ or $p_2$. If $\theta=0$, the integral curve stays in $\mathcal{R}_{FS}$ and hence converges to $p_1$. Since $\frac{Z_3}{Z_2}<1$ initially along the integral curve with $\theta>0$, the statement follows from Proposition \ref{prop_repellent in B 1}.
\end{proof}

For the negative Einstein metrics, the integral curve converges to some point in $P$ by Proposition \ref{prop_Limit Classification}. By each one of $a,b,c$ exhibits exponential growth since $\frac{X_i}{\sqrt{W}}$ at each point in $P$. We hence have the following lemma.
\begin{lemma}
\label{lem_asymp for PE}
Each $\xi(k,\theta,0,s_5)$ with $s_5>0$ is AH. 
\end{lemma}

\subsection{Non-Einstein Ricci solitons $(s_4>0)$}
Subtlety occurs for non-Einstein solitons, since the same limiting critical point $p_0$ may correspond to different asymptotic geometries, depending on the limiting ratios of the coordinate functions. For non-Einstein expanders ($s_4,s_5>0$), the asymptotic analysis in \cite[pp. 1122--1123]{dancer_non-kahler_2009} applies to all integral curves considered here. Hence we obtain the following lemma.

\begin{lemma}
\label{lem_asymp for expanders}
Each $\xi(k,\theta,s_4,s_5)$ with $s_4,s_5>0$ is AC. 
\end{lemma}

For steady Ricci solitons ($s_4>0$, $s_5=0$), the ratio analysis for integral curves entering $\mathcal{A}$ was established in \cite[Section 4]{chi_Non-Shrinking_2024}. Thus, we have
\begin{lemma}
\label{lem_asymp for steady in A}
Each $\xi(k,\theta,s_4,0)$ with $s_4>0$ that enters $\mathcal{A}$ is ACP. The base of the asymptotic paraboloid is the Fubini--Study $\mathbb{CP}^{2m+1}$ if $\theta=0$; non-K\"ahler $\mathbb{CP}^{2m+1}$ if $\theta\in(0,\pi)$.
\end{lemma}

The non-collapsed steady solitons arise from integral curves staying in $\mathcal B$, which is analyzed separately below.

\begin{proposition}
\label{prop_Z1 to 1 in B}
If a $\xi(k,\theta,s_4,0)$ with $s_4> 0$ remains in $\mathcal{B}$, the limits $\mu^2=\lim\limits_{\eta\to\infty}Z_1$ and $\nu^2=\lim\limits_{\eta\to\infty}\sqrt{\frac{Z_3}{Z_2}}$ exist. The possible values for $(\mu^2,\nu^2)$ are 
$$
(1,1),\quad \left(1,\frac{1}{2m+3}\right).
$$
\end{proposition}
\begin{proof}
Since $Z_1$ monotonically increases in the compact set $\mathcal{B}$, we must have $\mu^2\in (0,1]$. We can hence generalize \cite[Proposition 3.6]{wink2023cohomogeneity} and show that $\nu^2$ exists; see also \cite[(4.11)]{chi_Non-Shrinking_2024}. Possible values for $(\mu^2,\nu^2)$ are derived directly from \cite[Lemma 4.8]{chi_Non-Shrinking_2024}.
\end{proof}

The paraboloidal growth now follows from \cite[Lemma 4.8]{chi_Non-Shrinking_2024}, which establishes that
\begin{equation}
\label{eqn_X1Z1/Z2}
\begin{split}
\lim\limits_{t\to\infty}a\dot{a}&=\lim\limits_{\eta\to\infty}\frac{Q}{C}\frac{X_1Z_1}{Z_2}=\frac{1}{-C}\left(2+4 m \nu^4\right),\\
\lim\limits_{t\to\infty}b\dot{b}&=\lim\limits_{\eta\to\infty}\frac{Q}{C}\frac{X_2}{Z_2}=\frac{1}{-C}\left(2+4 m \nu^4\right),\\
\lim\limits_{t\to\infty}c\dot{c}&=\lim\limits_{\eta\to\infty}\frac{Q}{C}\frac{X_3}{Z_4}=\frac{1}{-C}\left((4m+8)-6 \nu^2 \right).
\end{split}
\end{equation}
and hence
\begin{equation}
\label{eqn_paraboloid asymtptotics}
\begin{split}
a^2&\sim \frac{1}{\sqrt{-C}}\left(2+ 4m \nu^4\right)t,\\
b^2&\sim \frac{1}{\sqrt{-C}}\left(2+4m \nu^4\right)t,\\
c^2&\sim  \frac{1}{\sqrt{-C}}\left((4m+8)-6\nu^2\right)t.
\end{split}
\end{equation}
It remains to identify the base of the asymptotic paraboloid in terms of the parameter $\theta$. The following proposition extends \cite[Proposition 4.10]{chi_Non-Shrinking_2024} from integral curves in $\mathcal{A}$ to those remaining in $\mathcal{B}$. We include the proof for completeness.

\begin{proposition}
\label{prop_repellent in B 2}
For a $\xi(k,\theta,s_4,0)$ with $s_4> 0$ that remains in $\mathcal{B}$ and $\sqrt{\frac{Z_3}{Z_2}}<1$ initially, if $\nu^2=\lim\limits_{\eta\to\infty}\sqrt{\frac{Z_3}{Z_2}}$ exists, it is necessary that $\nu^2<1$.
\end{proposition}
\begin{proof}
With the similar argument as Proposition \ref{prop_repellent in B 1}, it is necessary that $X_2-X_3$ being eventually positive given that $\nu^2=1$.

Rewrite \eqref{eqn_key derivative 1} as 
\begin{equation}
\label{eqn_key derivative 2}
\begin{split}
&\left(2(\sqrt{Z_2}-\sqrt{Z_3})+X_3-X_2\right)'\\
&=(X_2-X_3)(1-G-4\sqrt{Z_3})\\
&\quad +(\sqrt{Z_2}-\sqrt{Z_3})(2G-2X_2 +4(m+1)\sqrt{Z_3} - 4\sqrt{Z_2} +2Z_1(\sqrt{Z_2} +\sqrt{Z_3}))
\end{split}
\end{equation}
Since $X_2-X_3$ is eventually positive, the first term is eventually positive by Proposition \ref{prop_Limit Classification}. We have 
\begin{equation}
\begin{split}
&\left(2(\sqrt{Z_2}-\sqrt{Z_3})+X_3-X_2\right)'\\
&\geq (\sqrt{Z_2}-\sqrt{Z_3})(2G-2X_2 +4(m+1)\sqrt{Z_3} - 4\sqrt{Z_2})\\
&= (\sqrt{Z_2}-\sqrt{Z_3})\left(2G-2\frac{X_2}{Z_2}Z_2 +4(m+1)\sqrt{Z_3} - 4\sqrt{Z_2}\right).
\end{split}
\end{equation}
Since $Z_2,Z_3\to 0$, $\nu^2=1$, and $\frac{X_2}{Z_2}\to 2+4m\nu^2$ by \eqref{eqn_X1Z1/Z2}, the second factor is eventually positive. Therefore, the derivative $\left(2(\sqrt{Z_2}-\sqrt{Z_3})+X_3-X_2\right)'$ is eventually positive and does not converge to zero, a contradiction to (1) of Proposition \ref{prop_Limit Classification}.
\end{proof}

Therefore, we have the following lemma.
\begin{lemma}
\label{lem_asymp for steady in B}
Each $\xi(k,\theta,s_4,0)$ with $s_4>0$ that stays in $\mathcal{B}$ is AP. The base of the asymptotic paraboloid is the standard  $\mathbb{S}^{4m+3}/\mathbb{Z}_k$ if $\theta=0$; Jensen $\mathbb{S}^{4m+3}/\mathbb{Z}_k$ if $\theta\in(0,\pi)$.
\end{lemma}
\begin{proof}
The paraboloidal asymptotic is established in \eqref{eqn_paraboloid asymtptotics}. The base of the paraboloid is the standard $\mathbb{S}^{4m+3}/\mathbb{Z}_k$ ($\nu^2=1$) or the Jensen $\mathbb{S}^{4m+3}/\mathbb{Z}_k$ ($\nu^2=\frac{1}{2m+3}$). If $\theta=0$, the integral curve $\xi(k,0,s_4,0)$ stays in $\mathcal{RS}_{\text{FS}}$, where $X_2=X_3$ and $Z_2=Z_3$. Thus the base is round. Since $\frac{Z_3}{Z_2}<1$ initially along $\xi(k,\theta,s_4,0)$ with $\theta>0$, the limit $\nu^2< 1$ by Proposition \ref{prop_repellent in B 2}. The base of the asymptotic paraboloid is the Jensen $\mathbb{S}^{4m+3}/\mathbb{Z}_k$.
\end{proof}

We are ready to prove the main theorems.
\begin{proof}[Proof for Theorem \ref{thm_1}]
The completeness of Ricci solitons in Theorem \ref{thm_1} is proved by Lemma \ref{lem_long_existing xi_1}. 

We show that each $\xi(k,\theta,s_4,0)$ in Lemma \ref{lem_long_existing xi_1} enters $\mathcal{F}\cap\mathcal{A}$. Since $1-Z_1\geq 0$ holds in the invariant set $\mathcal{F}$, by Lemma \ref{lem_B, C or D}, each $\xi(k,\theta,s_4,0)$ in Lemma \ref{lem_long_existing xi_1} enters $\mathcal{A}$ or stays in $\mathcal{B}$. We show that the latter situation is impossible.

Suppose an integral curve in Lemma \ref{lem_long_existing xi_1} stays in $\mathcal{B}$. By \eqref{eqn_derivative of F2}, we have
\begin{equation}
\begin{split}
F_2'&=F_2G+(4\sqrt{Z_1Z_2}-1)F_2+(1-X_1-2\sqrt{Z_1Z_2})2\sqrt{\frac{Z_2}{Z_1}}(1-Z_1)+4mZ_3(1-Z_1)\\
&\geq (4\sqrt{Z_1Z_2}-1)F_2+(1-X_1-2\sqrt{Z_1Z_2})2\sqrt{\frac{Z_2}{Z_1}}(1-Z_1)\\
&=(4\sqrt{Z_1Z_2}-1)\left(X_2-X_1\right)+\left(1-\frac{1}{2}\sqrt{\frac{X_1Z_1}{Z_2}}\frac{\sqrt{X_1}}{Z_1}\right)4Z_2(1-Z_1)
\end{split}
\end{equation}
If a $\xi(k,\theta,0,0)$ stays in $\mathcal{F}\cap\mathcal{B}$, the integral curve converges to $p_1$ or $p_2$ by Proposition \ref{prop_Limit Classification}. As the first term vanishes, while the factor $\left(1-\frac{1}{2}\sqrt{\frac{X_1Z_1}{Z_2}}\frac{\sqrt{X_1}}{Z_1}\right)>0$ at both critical points, the derivative $F_2'$ is positive near $p_1$ and $p_2$ while the integral curve is in $\mathcal{A}$. But $F_2$ vanishes at $p_1$ and $p_2$. We reach a contradiction.

If a $\xi(k,\theta,s_4,0)$ with $s_4>0$ stays in $\mathcal{F}\cap\mathcal{B}$, the first term is eventually positive by Proposition \ref{prop_Limit Classification}. The second term is also eventually positive by \eqref{eqn_X1Z1/Z2}. Therefore, the derivative $F_2'$ is eventually positive. Since $F_2>0$ initially and the function vanishes at $(0,0,0,1,0,0,0,0)$, we reach a contradiction.

The asymptotics for steady solitons are thus established by Lemma \ref{lem_asymp for RF in A} and Lemma \ref{lem_asymp for steady in A}. The asymptotics for expanders are thus established by Lemma \ref{lem_asymp for PE} and Lemma \ref{lem_asymp for expanders}. 
\end{proof}

\begin{proof}[Proof for Theorem \ref{thm_2-1}-\ref{thm_3}]

For Theorem \ref{thm_2-1}-\ref{thm_2-2}, the completeness is proved by Lemma \ref{lem_long exiting xik}. he asymptotics for expanders are thus established by Lemma \ref{lem_asymp for PE} and Lemma \ref{lem_asymp for expanders}. For steady solitons, if the integral curve enters $\mathcal{A}$, its asymptotic is either ACP by Lemma \ref{lem_asymp for steady in A}, or ALC by Lemma \ref{lem_asymp for RF in A}. If the integral curve enters $\mathcal{B}$, the Ricci soliton is non-collapsed. Its asymptotic is either AP by Lemma \ref{lem_asymp for steady in B}, or AC by Lemma \ref{lem_asymp for RF in B}.

To prove Theorem \ref{thm_3}, by Lemma \ref{lem_asymp for RF in B} and Lemma \ref{lem_asymp for steady in B}, it suffices to prove the existence of some $\xi(k,\theta,s_4,0)$ and $\xi(k,\theta,0,0)$ that stay in $\mathcal{B}$. This is immediate by Lemma \ref{lem_for thm 1.4}.
\end{proof}

\bibliography{Bibliography}
\bibliographystyle{alpha}
\end{document}